\documentclass[a4paper,10pt]{article}

\usepackage{a4wide}
\usepackage[dvips]{epsfig}
\usepackage{amsmath,amssymb,amsthm,xspace,float}
\usepackage{amsfonts}
\usepackage{color}
\usepackage[dvips]{graphicx}
\usepackage{amsmath, amssymb, amsthm}
\usepackage{tikz}
\usepackage{mathtools}
\usepackage{bm}
\usepackage{url}

\newtheorem{thm}{Theorem}
\newtheorem{definition}[thm]{Definition}
\newtheorem{lemma}[thm]{Lemma}
\newtheorem{proposition}[thm]{Proposition}

\newtheorem{corollary}[thm]{Corollary}

\theoremstyle{definition}
\newtheorem{remark}[thm]{Remark}

\newcommand{\commonpat}{2\underbracket[.5pt][1pt]{41}3}
\newcommand{\baxpat}{3\underbracket[.5pt][1pt]{14}2}
\newcommand{\twistedpat}{3\underbracket[.5pt][1pt]{41}2}
\newcommand{\planepat}{2\underbracket[.5pt][1pt]{14}3}

\title{Semi-Baxter and strong-Baxter: \\ two relatives of the Baxter sequence}
\author{Mathilde Bouvel, Veronica Guerrini, Andrew Rechnitzer, and Simone Rinaldi}
\date{\today}

\begin{document}

\maketitle

\begin{abstract}
In this paper, we enumerate two families of pattern-avoiding permutations: 
those avoiding the vincular pattern $2\underbracket[.5pt][1pt]{41}3$, which we call semi-Baxter permutations, 
and those avoiding the vincular patterns $2\underbracket[.5pt][1pt]{41}3$, $3\underbracket[.5pt][1pt]{14}2$ and $3\underbracket[.5pt][1pt]{41}2$, 
which we call strong-Baxter permutations. 
We call semi-Baxter numbers and strong-Baxter numbers the associated enumeration sequences. 
We prove that the semi-Baxter numbers enumerate in addition plane permutations (avoiding $2\underbracket[.5pt][1pt]{14}3$). The problem of counting these permutations was open 
and has given rise to several conjectures, which we also prove in this paper. 

For each family (that of semi-Baxter -- or equivalently, plane -- and that of strong-Baxter permutations), we describe a generating tree, 
which translates into a functional equation for the generating function. 
For semi-Baxter permutations, it is solved using (a variant of) the kernel method: 
this gives an expression for the generating function while also proving its D-finiteness. 
From the obtained generating function, we derive closed formulas for the semi-Baxter numbers, a recurrence that they satisfy, as well as their asymptotic behavior. 
For strong-Baxter permutations, we show that their generating function is (a slight modification of) 
that of a family of walks in the quarter plane, which is known to be non D-finite.
\end{abstract}

\section{Introduction}

The purpose of this article is the study of two enumeration sequences, which we call the \emph{semi-Baxter sequence} and the \emph{strong-Baxter sequence}. 
They enumerate, among other objects, families of pattern-avoiding permutations 
closely related to the well-known family of Baxter permutations, and to the slightly less popular one of twisted Baxter permutations, 
which are both counted by the sequence of Baxter numbers~\cite[sequence A001181]{OEIS}. 

Recall that a permutation $\pi = \pi_1 \pi_2 \dots \pi_n$ contains the vincular\footnote{
Throughout the article, we adopt the convention of denoting by the symbol $\underbracket[.5pt][1pt]{~~~}$ the elements 
that are required to be adjacent in an occurrence of a vincular pattern, rather than using the historical notation with dashes 
wherever elements are not required to be consecutive. For instance, our pattern $\commonpat$ is sometimes written $2-41-3$ in the literature.} 
pattern $\commonpat$ 
if there exists a subsequence $\pi_i \pi_j \pi_{j+1} \pi_k$ of $\pi$ (with $i<j<k-1$), 
called an \emph{occurrence} of the pattern, 
that satisfies $\pi_{j+1} < \pi_i < \pi_k < \pi_j$. 
Containment and occurrences of the patterns $\baxpat$, $\twistedpat$, $2\underbracket[.5pt][1pt]{14}3$ and $\underbracket[.5pt][1pt]{14}23$ are defined similarly. 
A permutation not containing a pattern avoids it. 
Baxter permutations~\cite[among many others]{BM} are those that avoid both $\commonpat$ and $\baxpat$, 
while twisted Baxter permutations~\cite[and references therein]{Twis} are the ones avoiding $\commonpat$ and $\twistedpat$. 
We denote by $Av(P)$ the family of permutations avoiding all patterns in $P$. 

The two sequences that will be our main focus are first the one enumerating permutations avoiding $\commonpat$, called \emph{semi-Baxter permutations}, 
and second the one enumerating permutations avoiding all three patterns $\commonpat$, $\baxpat$ and $\twistedpat$, called \emph{strong-Baxter permutations}.
Remark that a permutation avoiding the (classical) pattern $231$ necessarily avoids $\commonpat$, $\baxpat$ and $\twistedpat$, 
and recall that $Av(231)$ is enumerated by the sequence of Catalan numbers. 
Therefore, the definitions in terms of pattern-avoidance and the enumeration results given above can be summarized as shown in Figure~\ref{fig:inclusions}. 

\begin{figure}[ht]
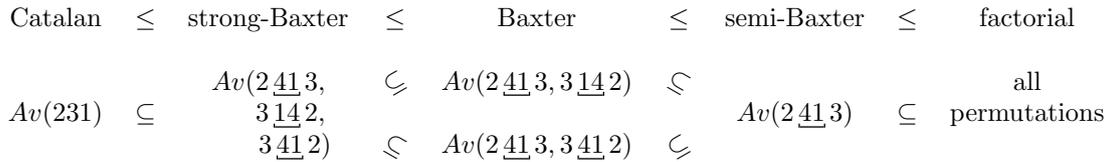

\begin{center}
\begin{tabular}{ccccccccc}
 Catalan & $\leq$ & strong-Baxter & $\leq$ & Baxter & $\leq$ & semi-Baxter & $\leq$ & factorial \\
  & &  &  &  & &  &  & \\
  & & $Av(\commonpat,$ & \rotatebox[origin=c]{30}{$\subseteq$} & $Av(\commonpat,\baxpat)$ & \rotatebox[origin=c]{-30}{$\subseteq$} &  &  & all \\
 $Av(231)$ & $\subseteq$ & ~~~~~$\baxpat,$ & & & & $Av(\commonpat)$ & $\subseteq$ & permutations \\
  & & ~~~~~~$\twistedpat)$ & \rotatebox[origin=c]{-30}{$\subseteq$} & $Av(\commonpat,\twistedpat)$ & \rotatebox[origin=c]{30}{$\subseteq$} &  &  & \\
\end{tabular}
\end{center}
\caption{Sequences from Catalan to factorial numbers, with nested families of pattern-avoiding permutations that they enumerate. }
\label{fig:inclusions}
\end{figure}

\bigskip

The focus of this paper is the study of the two sequences of semi-Baxter and strong-Baxter numbers. 

\smallskip

We deal with the semi-Baxter sequence (enumerating semi-Baxter permutations) in Section~\ref{sec:semi}. 
It has been proved in~\cite{Kasraoui} (as a special case of a general statement) that this sequence also enumerates \emph{plane permutations}, 
defined by the avoidance of $2\underbracket[.5pt][1pt]{14}3$. This sequence is referenced as A117106 in~\cite{OEIS}. 
We first give a more specific proof that plane permutations and semi-Baxter permutations are equinumerous, 
by providing a common generating tree (or succession rule) with two labels
for these two families. Basics and references about generating trees can be found in Section~\ref{sec:review_gentree}. 

We solve completely the problem of enumerating semi-Baxter permutations (or equivalently, plane permutations), 
pushing further the techniques that were used to enumerate Baxter permutations in~\cite{BM}. 
Namely, we start from the functional equation associated with our succession rule for semi-Baxter permutations, 
and we solve it using variants of the kernel method \cite{BM,iteratedKM}. 
This results in an expression for the generating function for semi-Baxter permutations, 
showing that this generating function is D-finite\footnote{Recall that $F(x)$
is D-finite when there exist $k \geq 0$ and polynomials $Q(x), Q_0(x), \dots, Q_k(x)$ of $\mathbb{Q}[x]$ with $Q_k(x) \neq 0$ such that 
$Q_0(x) F(x) + Q_1(x) F'(x) + Q_2(x) F''(x) + \dots + Q_k(x) F^{(k)}(x) = Q(x)$.}. 
From it, we obtain several formulas for the semi-Baxter numbers: 
first, a complicated closed formula;
second, a simple recursive formula; 
and third, three simple closed formulas that were conjectured by D. Bevan~\cite{BevanPrivate}. 

The problem of enumerating plane permutations was posed by M.~Bousquet-M\'elou and S.~Butler in~\cite{BB}. 
Some conjectures related to this enumeration problem were later proposed, in particular by D.~Bevan~\cite{bevan,BevanPrivate} and M.~Martinez and C.~Savage~\cite{savage}. 
Not only do we solve the problem of enumerating plane permutations (or equivalently, semi-Baxter permutations) completely,
but we also prove these conjectures. 
In addition, from one of these (former) conjectures (relating the semi-Baxter sequence to sequence \text{A005258} of~\cite{OEIS}, whose terms are sometimes called Ap\'ery numbers),
we easily deduce the asymptotic behavior of semi-Baxter numbers.

We mention that it has been conjectured in~\cite{BaxterShattuck} by A.~Baxter and M.~Shattuck that permutations avoiding $\underbracket[.5pt][1pt]{14}23$ 
are also enumerated by the same sequence, but we have not been able to prove it. 

\smallskip

In Section~\ref{sec:strong}, we focus on the study of strong-Baxter permutations and of the strong-Baxter sequence. 
Again, we provide a generating tree for strong-Baxter permutations, and translate the corresponding succession rule into a functional equation for their generating function. 
However, we do not solve the equation using the kernel method. 
Instead, from the functional equation, we prove that the generating function for strong-Baxter permutations is a very close relative 
of the one for a family of walks in the quarter plane studied in~\cite{bostan}.
As a consequence, the generating function for strong-Baxter permutations is not D-finite. 
Families of permutations with non D-finite generating functions are quite rare in the literature on pattern-avoiding permutations 
(although mostly studied for classical patterns, instead of vincular ones -- see the analysis in~\cite{nonDF1,nonDF2}): 
this makes the example of strong-Baxter permutations particularly interesting.

\bigskip

The article is next organized as follows. 
Section~\ref{sec:review_gentree} recalls easy facts about the Catalan sequence, 
and includes basics about generating trees and succession rules. 
Sections~\ref{sec:semi}, \ref{sec:Bax} and \ref{sec:strong} then focus on the sequences of semi-Baxter numbers, Baxter numbers, and strong Baxter numbers, respectively, 
and on the associated families of pattern-avoiding permutations.

\section{The Catalan family $Av(231)$ and a Catalan succession rule}\label{sec:review_gentree}

Generating trees and succession rules will be important for our work. 
We give a brief general presentation below. Details can be found for instance in~\cite{GFGT,Eco,BM,West_gt}. 
We also review the classical succession rule for Catalan numbers. 
This rule encodes generating trees for many Catalan families (see \cite{Eco}), 
but we will present only a generating tree for the family of permutations avoiding $231$, 
since we will build on it later in this work. 

\medskip

Consider any combinatorial class $\mathcal{C}$, that is to say any set of discrete objects equipped with a notion of size, 
such that there is a finite number of objects of size $n$ for any integer $n$. Assume also that $\mathcal{C}$ contains exactly one object of size $1$. 
A \emph{generating tree} for $\mathcal{C}$ is an infinite rooted tree, whose vertices are the objects of $\mathcal{C}$, 
each appearing exactly once in the tree, 
and such that objects of size $n$ are at level $n$ in the tree, that is to say at distance $n-1$ from the root 
(thus, the root is at level $1$, its children are at level $2$, and so on). 
The children of some object $c \in \mathcal{C}$ are obtained by adding an \emph{atom} (\emph{i.e.}~a piece of object that makes its size increase by $1$) to $c$. 
Of course, since every object should appear only once in the tree, not all additions are possible. 
We should ensure the unique appearance property by considering only additions that follow some restricted rules. 
We will call the \emph{growth} of $\mathcal{C}$ the process of adding atoms following these prescribed rules. 

Our focus in this section is on $Av(231)$, the set of permutations avoiding the pattern $231$:
a permutation $\pi$ avoids $231$ when it does not contain any subsequence $\pi_i \pi_j \pi_k$ (with $i<j<k$) such that $\pi_k < \pi_i < \pi_j$. 
A growth for $Av(132)$ has been originally described in~\cite{West_gt}, by insertion of a maximal element, which can be translated by symmetry 
into a growth for $Av(231)$ by insertion of a maximal or a leftmost element. 
In our paper, we are however interested in a different (and not symmetric) growth for $Av(231)$: by insertion of a rightmost element, in the same flavor as what is done in~\cite{BFR} for subclasses of $Av(231)$. 

Indeed, throughout the paper our permutations will grow by performing ``local expansions'' on the right of any permutation $\pi$. 
More precisely, when inserting $a \in \{1, \dots, n+1\}$ on the right of any $\pi$ of size $n$, 
we obtain the permutation $\pi' = \pi'_1 \dots \pi'_{n}\pi'_{n+1}$ where $\pi'_{n+1}=a$, $\pi'_i = \pi_i$ if $\pi_i < a$ and $\pi'_i = \pi_i +1$ if $\pi_i \geq a$. 
We use the notation $\pi \cdot a$ to denote $\pi'$. 
For instance, $1\,4\,2\,3 \cdot 3 = 1\,5\,2\,4\,3$.
This is easily understood on the diagrams representing permutations (which consist of points in the Cartesian plane at coordinates $(i,\pi_i)$): 
a local expansion corresponds to adding a new point on the right of the diagram, which lies vertically between two existing points (or below the lowest, or above the highest), 
and finally normalizing the picture obtained -- see Figure~\ref{fig:av231}.
These places where new elements may be inserted are called \emph{sites} of the permutation. 

\begin{figure}[ht]
\begin{center}
\includegraphics[scale=0.75]{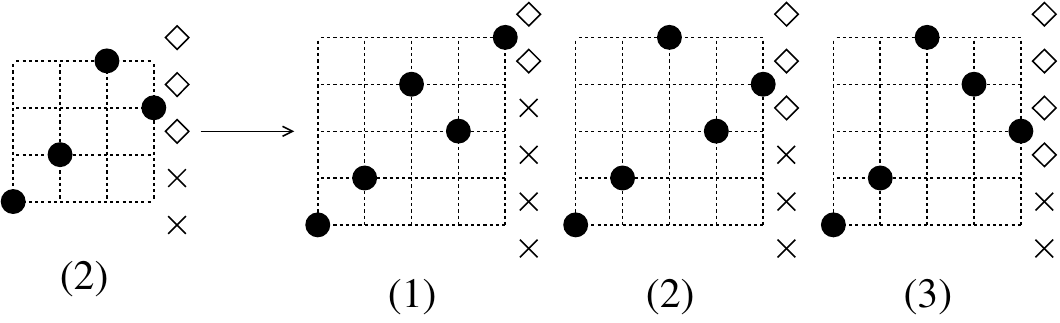}
\end{center}
\caption{
The growth of a permutation avoiding $231$: active sites are marked with $\Diamond$ and non-active sites by $\times$.}
\label{fig:av231}
\end{figure}

Clearly, performing this growth without restriction on the values $a$ would produce a generating tree for the family of all permutations. 
To ensure that only permutations avoiding $231$ appear in the tree (that all such appear exactly once being then obvious), 
insertions are not possible in all sites, but only in those such that the insertion does not create an occurrence of $231$ -- see Figure~\ref{fig:av231}. 
Such sites are called \emph{active sites}. 
In the considered example, the active sites of $\pi \in Av(231)$ are easily characterized as those above the largest element $\pi_i$ such that
there exists $j$ with $i<j$ and $\pi_i<\pi_j$. 
The first few levels of the generating tree for $Av(231)$ are shown in Figure~\ref{fig:gen_tree} (left). 

\begin{figure}[ht]
\begin{center}
\includegraphics[scale=0.45]{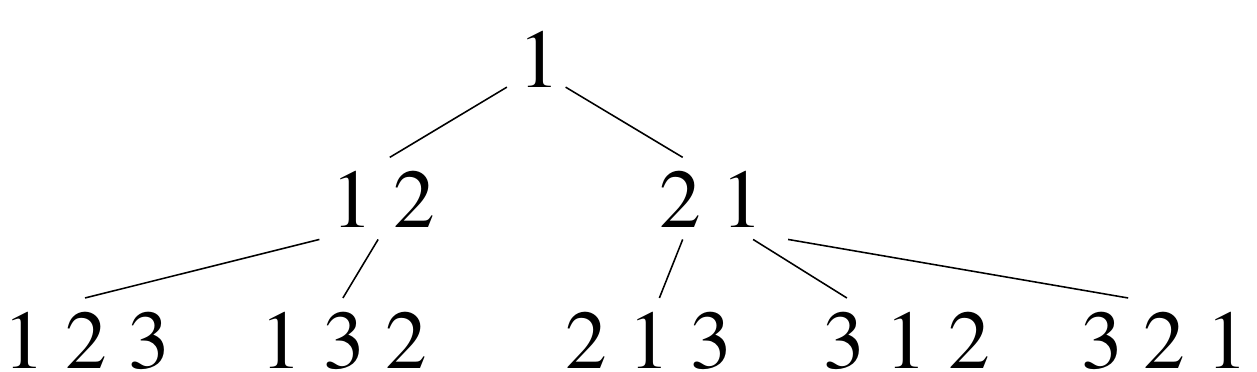} \qquad \includegraphics[scale=0.4]{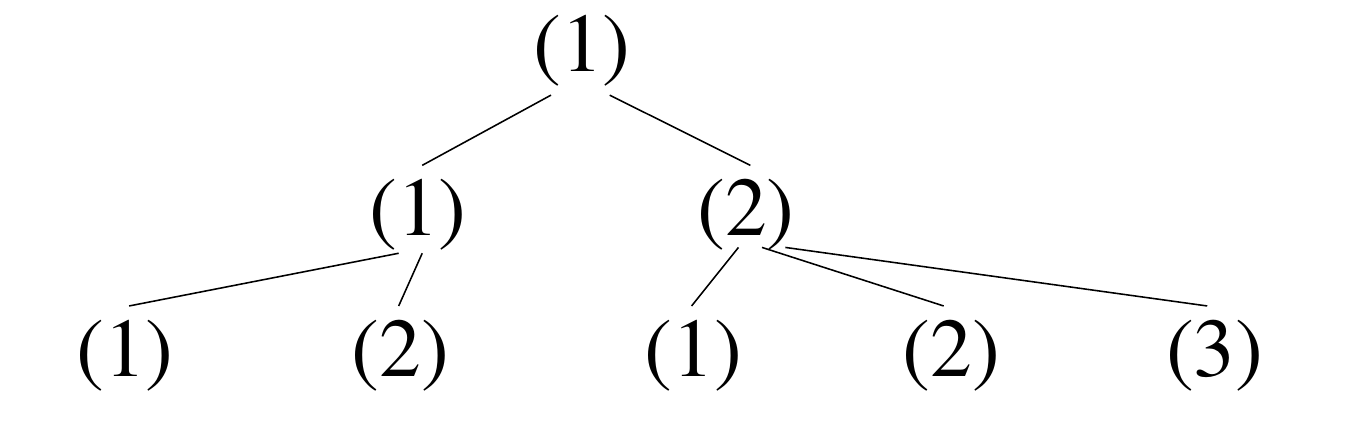} 
\end{center}
\caption{Two ways of looking at the generating tree for $Av(231)$: with objects (left) and with labels from the succession rule $\Omega_{Cat}$ (right).}
\label{fig:gen_tree}
\end{figure}

\medskip

Of importance for enumeration purposes is the general shape of a generating tree, not the specific objects labeling its nodes. 
From now on, when we write generating tree, we intend this shape of the tree, without the objects labeling the nodes. 
A \emph{succession rule} is a compact way of representing such a generating tree for a combinatorial class $\mathcal{C}$ without referring to its objects, but identifying them with \emph{labels}. 
Therefore, a succession rule is made of one starting label corresponding to the label of the root, and of productions encoding the way labels spread in the generating tree. 
From the beginnings~\cite{West_gt}, generating trees and succession rules have been used to derive enumerative results.
As we explain in~\cite{paper1}, the sequence enumerating the class $\mathcal{C}$ can be recovered from the succession rule itself, 
without reference to the specifics of the objects in $\mathcal{C}$: 
indeed, the $n$th term of the sequence is the total number of labels (counted with repetition) that are produced from the root by $n-1$ applications of the productions, 
or equivalently, the number of nodes at level $n$ in the generating tree. 

From the growth for $Av(231)$ described above, 
examining carefully how the number of active sites evolves when performing insertions (as done in~\cite{West_gt} or~\cite{BFR} for example), 
we obtain the following (and classical) succession rule associated with Catalan numbers 
(corresponding to the tree shown in Figure~\ref{fig:gen_tree}, right): 
$$\Omega_{Cat}=\left\{\begin{array}{ll}
(1)\\
(k) \rightsquigarrow (1), (2), \dots , (k), (k+1). \end{array}\right.$$
The intended meaning of the label $(k)$ is the number of active sites of a permutation, minus $1$. 
In Figure~\ref{fig:av231} and similar figures later on, the labels of the permutations are indicated below them.

\section{Semi-Baxter numbers}\label{sec:semi}
\subsection{Definition, context, and summary of our results} \label{sec:intro_semiBax} 

\begin{definition}
\label{dfn:semiBaxPerm}
A \emph{semi-Baxter permutation} is a permutation that avoids the pattern $\commonpat$. 
\end{definition}

\begin{definition}
\label{dfn:semiBaxNumber}
The sequence of \emph{semi-Baxter numbers}, $(SB_n)$, is defined by taking $SB_n$ to be the number of semi-Baxter permutations of size $n$. 
\end{definition}

The name ``semi-Baxter'' has been chosen because $\commonpat$ is one of the two patterns (namely, $\commonpat$ and $\baxpat$) whose avoidance defines the family of so-called Baxter permutations~\cite{CGHK78,Gire}, 
enumerated by the Baxter numbers~\cite[sequence \text{A001181}]{OEIS}. 
(Remark that up to symmetry, we could have defined semi-Baxter permutations by the avoidance of $\baxpat$, obtaining the same sequence.) 
Note that $\commonpat$ is also one of the two patterns (namely, $\commonpat$ and $\twistedpat$) whose avoidance defines the family of so-called twisted Baxter permutations~\cite{Rea05,West}, also enumerated by the Baxter numbers. 

The first few terms of the sequence of semi-Baxter numbers are 
\[1,2,6,23,104,530,2958,17734,112657, 750726, 5207910, 37387881, 276467208, \ldots\]

The family of semi-Baxter permutations already appears in the literature, at least on a few occasions. 
Indeed, it is an easy exercise to see that the avoidance of $\commonpat$ is equivalent to that of the barred pattern $25\bar{3}14$, which has been studied by L. Pudwell in~\cite{Pudwell}. 
(The definition of barred patterns, which is not essential to our work, can be found in~\cite{Pudwell}.)
In this work, by means of enumeration schemes L. Pudwell suggests that the enumerative sequences of semi-Baxter permutations and \emph{plane permutations} 
(see Definition~\ref{dfn:planePerm} below) coincide. 
This conjecture has later been proved as a special case of a general statement in~\cite[Corollary 1.9(b)]{Kasraoui}.
In Section~\ref{sec:GenTreeSemiBax} we give an alternative and self-contained proof that plane permutations and semi-Baxter permutations are indeed equinumerous.
The sequence enumerating plane permutations has already been registered on the OEIS~\cite{OEIS} as sequence \text{A117106}, which is then our sequence $(SB_n)$. 

The enumeration of plane permutations has received a fair amount of attention in the literature. 
It first arose as an open problem in~\cite{BB}. This family of permutations, indeed, was identified as a superset of forest-like permutations, which are thoroughly investigated in~\cite{BB}. 
A forest-like permutation is any permutation whose Hasse graph is a forest -- the Hasse graph of a permutation $\pi$ of size $n$ is the oriented graph on the vertex set $\{1,\ldots,n\}$, 
which includes an edge from $i$ to $j$ (for $i<j$) if and only if $\pi(i)<\pi(j)$ and there is no $k$ such that $i<k<j$ and $\pi(i)<\pi(k)<\pi(j)$, 
and with all edges pointing upward. 
For instance, the Hasse graphs of permutations $2413$ and $2143$ are depicted in Figure~\ref{fig:Hasse}. 
In addition, it shows that the Hasse graph of $2413$ is plane (\emph{i.e.} can be drawn on the plane without any crossing of edges), while the one of $2143$ is not.
\begin{figure}[ht]
\centering
\begin{tikzpicture}[scale=0.5]
\begin{scope}
\node  at (-.5,-.5) {$1$};
\node at (-.5,2.5) {$2$};
\node at (2.5,2.5) {$4$};
\node at (2.5,-.5) {$3$};
\filldraw[black] (0,0) circle (2pt); 
\filldraw[black] (2,0) circle (2pt); 
\filldraw[black] (2,2) circle (2pt); 
\filldraw[black] (0,2) circle (2pt); 
\draw[->] (0,0) -- (0,2);
\draw[->] (0,0) -- (2,2);
\draw[->] (2,0) -- (2,2);
\end{scope}\end{tikzpicture}\hspace{2.5cm}
\begin{tikzpicture}[scale=0.5]
\begin{scope}
\node  at (-.5,-.5) {$1$};
\node at (-.5,2.5) {$3$};
\node at (2.5,2.5) {$4$};
\node at (2.5,-.5) {$2$};
\filldraw[black] (0,0) circle (2pt); 
\filldraw[black] (2,0) circle (2pt); 
\filldraw[black] (2,2) circle (2pt); 
\filldraw[black] (0,2) circle (2pt); 
\draw[->] (0,0) -- (0,2);
\draw[->] (0,0) -- (2,2);
\draw[->] (2,0) -- (0,2);
\draw[->] (2,0) -- (2,2);
\end{scope}
\end{tikzpicture}
\caption{\label{fig:Hasse}The Hasse graphs of permutations $2413$ (left) and $2143$ (right).}
\end{figure}
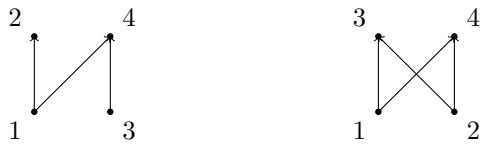

The authors of~\cite{BB} named plane permutations those permutations whose Hasse graph is plane, 
characterized them as those avoiding $\planepat$, and called for their enumeration.
This enumerative problem was studied with a quite experimental perspective, as one case of many, through enumeration schemes by L.~Pudwell in~\cite{Pudwell}. 
Then, D.~Bevan computed the first 37 terms of their enumerative sequence~\cite{bevan}, by iterating a functional equation provided in \cite[Theorem 13.1]{bevan}. 
Although~\cite{bevan} gives a functional equation for the generating function of semi-Baxter numbers, 
there is no formula (closed or recursive) for $SB_n$. 
There is however a conjectured explicit formula, which, in addition, gives information about their asymptotic behavior (see Proposition~\ref{prop:conj} and Corollary~\ref{cor:asym_SB_n}).
Another recursive formula for $SB_n$ has been conjectured by M.~Martinez and C.~Savage in~\cite{savage}, in relation with \emph{inversion sequences} avoiding some patterns 
(definition and precise statement are provided in Subsection~\ref{sec:inv}). 
Finally, closed formulas for $SB_n$ have been conjectured by D.~Bevan in~\cite{BevanPrivate}. 

\medskip

Our results about semi-Baxter numbers are the following. 
Most importantly, we solve the problem of enumerating semi-Baxter permutations, 
as well as plane permutations. 
We provide a common succession rule that governs their growth, presented in Subsection~\ref{sec:GenTreeSemiBax}. 
Next, in Subsection~\ref{sec:inv}, we show that inversion sequences avoiding the patterns $210$ and $100$ 
grow along the same rule, thereby proving a first formula for $SB_n$ and settling a conjecture  of~\cite{savage}.
Then, by means of standard tools we translate the succession rule into a functional equation whose solution is the generating function of semi-Baxter numbers. Subsection~\ref{sec:resultsGF} gives a closed expression for the generating function of semi-Baxter numbers, 
together with closed, recursive and asymptotic formulas for $SB_n$. 
The results of this subsection are proved in Subsection~\ref{sec:proofs_GF} following the same method as in~\cite{MBM_Xin}: 
the functional equation is solved using the obstinate kernel method, 
a first closed formula for $SB_n$ is obtained by the Lagrange inversion, 
the recursive formula follows from it applying the method of \emph{creative telescoping}~\cite{zeilberger}, 
which can then be applied again to prove that the explicit formulas for $SB_n$ conjectured in~\cite{BevanPrivate} are correct.
Finally, we prove the formula for $SB_n$ conjectured in~\cite{bevan}, which in turn gives us the asymptotic behavior of $SB_n$. 

\subsection{Succession rule for semi-Baxter permutations and plane permutations}\label{sec:GenTreeSemiBax}

Similarly to the case of permutations avoiding $231$ described in the introduction,
we will provide below generating trees for semi-Baxter permutations and for plane permutations, 
where permutations grow by insertion of an element on the right. 
Recall that for any permutation $\pi$ of size $n$ and for any $a \in \{1, \dots, n+1\}$, $\pi \cdot a$ denotes the permutation 
$\pi'$ of size $n+1$ such that $\pi'_{n+1}=a$, $\pi'_i = \pi_i$ if $\pi_i < a$ and $\pi'_i = \pi_i +1$ if $\pi_i \geq a$. 

\begin{proposition}
\label{prop:semi_rule}
A generating tree for semi-Baxter permutations can be obtained by insertions on the right, 
and it is isomorphic to the tree generated by the following succession rule: 
$$\Omega_{semi}=\left\{\begin{array}{ll}
(1,1)\\
(h,k) \rightsquigarrow \hspace{-3mm}& (1,k+1), \dots , (h,k+1)\\
& (h+k,1), \dots , (h+1,k). \end{array}\right.$$
\end{proposition}

\begin{proof}
First, observe that removing the last element of a permutation avoiding $\commonpat$, we obtain a permutation that still avoids $\commonpat$. 
So, a generating tree for semi-Baxter permutations can be obtained with local expansions on the right. 

For $\pi$ a semi-Baxter permutation of size $n$, the \emph{active sites} are by definition 
the points $a$ (or equivalently the values $a$) such that $\pi \cdot a$ is also semi-Baxter, \emph{i.e.}, avoids $\commonpat$. 
The other points $a$ are called non-active sites. 
An occurrence of $2\underbracket[.5pt][1pt]{31}$ in $\pi$ is a subsequence $\pi_j \pi_i \pi_{i+1}$ (with $j<i$) such that $\pi_{i+1}<\pi_j<\pi_i$. 
Obviously, the non-active sites $a$ of $\pi$ are characterized by the fact that $a \in (\pi_j,\pi_i]$ for some occurrence $\pi_j \pi_i \pi_{i+1}$ of $2\underbracket[.5pt][1pt]{31}$. 
We call a \emph{non-empty descent} of $\pi$ a pair $\pi_i \pi_{i+1}$ such that there exists $\pi_j$ that makes $\pi_j \pi_i \pi_{i+1}$ an occurrence of $2\underbracket[.5pt][1pt]{31}$. 
Note that in the case where $\pi_{n-1} \pi_n$ is a non-empty descent, choosing $\pi_j = \pi_{n} +1$ always gives an occurrence of $2\underbracket[.5pt][1pt]{31}$, 
and it is the smallest possible value of $\pi_j$ for which $\pi_j \pi_{n-1} \pi_n$ is an occurrence of $2\underbracket[.5pt][1pt]{31}$. 

To each semi-Baxter permutation $\pi$ of size $n$, we assign a label $(h,k)$, 
where $h$ (resp. $k$) is the number of the active sites of $\pi$ smaller than or equal to (resp. greater than) $\pi_n$. 
Remark that $h,k\geq1$, since $1$ and $n+1$ are always active sites. 
Moreover, the label of the permutation $\pi=1$ is $(1,1)$, which is the root in $\Omega_{semi}$.

Consider a semi-Baxter permutation $\pi$ of size $n$ and label $(h,k)$. 
Proving Proposition~\ref{prop:semi_rule} amounts to showing that permutations $\pi \cdot a$ have labels $(1,k+1), \dots , (h,k+1), (h+k,1), \dots , (h+1,k)$ 
when $a$ runs over all active sites of $\pi$. 
Figure~\ref{fig:semi}, which shows an example of semi-Baxter permutation $\pi$ with label $(2,2)$ and all the corresponding $\pi \cdot a$ with their labels, 
should help understanding the case analysis that follows. 
Let $a$ be an active site of $\pi$. 

Assume first that $a > \pi_n$ (this happens exactly $k$ times), so that $\pi \cdot a$ ends with an ascent. 
The occurrences of $2\underbracket[.5pt][1pt]{31}$ in $\pi \cdot a$ are the same as in $\pi$. 
Consequently, the active sites are not modified, except that the active site $a$ of $\pi$ is now split into two actives sites of $\pi \cdot a$: 
one immediately below $a$ and one immediately above. 
It follows that $\pi \cdot a$ has label $(h+k+1-i,i)$, if $a$ is the $i$-th active site from the top. 
Since $i$ ranges from $1$ to $k$, this gives the second row of the production of $\Omega_{semi}$. 

Assume next that $a=\pi_n$. Then, $\pi \cdot a$ ends with a descent, but an empty one. 
Similarly to the above case, we therefore get one more active site in $\pi \cdot a$ than in $\pi$, 
and $\pi \cdot a$ has label $(h,k+1)$, the last label in the first row of the production of $\Omega_{semi}$. 

Finally, assume that $a < \pi_n$ (this happens exactly $h-1$ times). 
Now, $\pi \cdot a$ ends with a non-empty descent, which is $(\pi_n +1) a$. 
It follows from the discussion at the beginning of this proof that all sites of $\pi \cdot a$ in $(a+1,\pi_n+1]$ become non-active, 
while all others remain active if they were so in $\pi$ (again, with $a$ replaced by two active sites surrounding it, one below it and one above). 
If $a$ is the $i$-th active site from the bottom, it follows that $\pi \cdot a$ has label $(i,k+1)$, hence giving all missing labels in the first row of the production of $\Omega_{semi}$. 
\end{proof}

\begin{figure}[ht]
\centering
\includegraphics[width=0.8\textwidth]{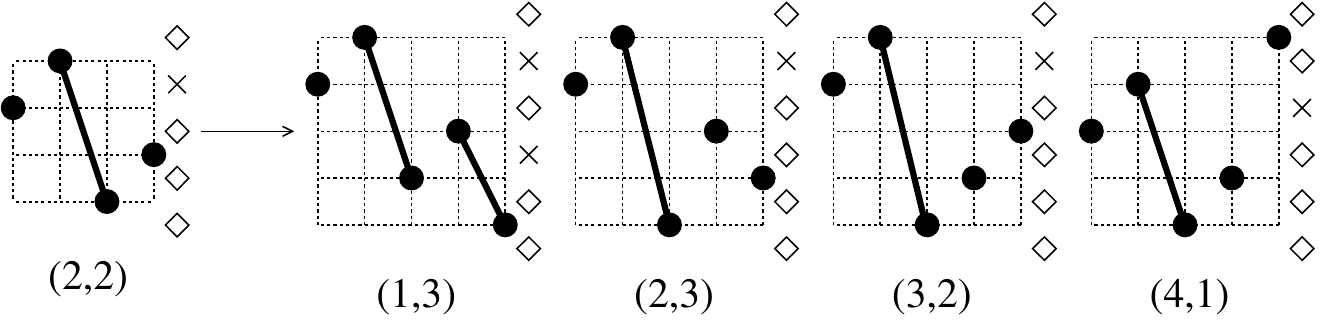}
\caption{The growth of semi-Baxter permutations (with notation as in Figure~\ref{fig:av231}). Non-empty descents are represented with bold lines.}
\label{fig:semi}
\end{figure}

\begin{definition}
A \emph{plane permutation} is a permutation that avoids the vincular pattern $2\underbracket[.5pt][1pt]{14}3$ (or equivalently, the barred pattern $21\bar{3}54$). 
\label{dfn:planePerm}
\end{definition}

\begin{proposition}\label{prop:plane}
A generating tree for plane permutations can be obtained by insertions on the right, 
and it is isomorphic to the tree generated by the succession rule $\Omega_{semi}$.
\end{proposition}

\begin{proof}
The proof of this statement follows applying the same steps as in the proof of Proposition~\ref{prop:semi_rule}. 
First, observe that removing the last element of a permutation avoiding $\planepat$, we obtain a permutation that still avoids $\planepat$. 
So, a generating tree for plane permutations can be obtained with local expansions on the right.

For $\pi$ a plane permutation of size $n$, the active sites are by definition 
the values $a$ such that $\pi \cdot a$ is also plane, \emph{i.e.}, avoids $\planepat$. The other points $a$ are called non-active sites. 
An occurrence of $2\underbracket[.5pt][1pt]{13}$ in $\pi$ is a subsequence $\pi_j \pi_i \pi_{i+1}$ (with $j<i$) such that $\pi_i<\pi_j<\pi_{i+1}$. 
Note that the non-active sites $a$ of $\pi$ are characterized by the fact that $a \in (\pi_j,\pi_{i+1}]$ for some occurrence $\pi_j \pi_i \pi_{i+1}$ of $2\underbracket[.5pt][1pt]{13}$. 
We call a \emph{non-empty ascent} of $\pi$ a pair $\pi_i \pi_{i+1}$ such that there exists $\pi_j$ that makes $\pi_j \pi_i \pi_{i+1}$ an occurrence of $2\underbracket[.5pt][1pt]{13}$. 
As in the proof of Proposition~\ref{prop:semi_rule}, if $\pi_{n-1} \pi_n$ is a non-empty ascent, $\pi_j = \pi_{n-1} +1$ is the smallest value of $\pi_j$ such that $\pi_j \pi_{n-1} \pi_n$ is an occurrence of $2\underbracket[.5pt][1pt]{13}$. 

Now, to each plane permutation $\pi$ of size $n$, we assign a label $(h,k)$, 
where $h$ (resp. $k$) is the number of the active sites of $\pi$ greater than (resp. smaller than or equal to) $\pi_n$.
Remark that $h,k\geq1$, since $1$ and $n+1$ are always active sites. 
Moreover, the label of the permutation $\pi=1$ is $(1,1)$, which is the root in $\Omega_{semi}$. 
The proof is concluded by showing that the permutations $\pi \cdot a$ have labels $(1,k+1), \dots , (h,k+1), (h+k,1), \dots , (h+1,k)$,
when $a$ runs over all active sites of $\pi$. 

If $a \leq \pi_n$, $\pi \cdot a$ ends with a descent, and 
it follows as in the proof of Proposition~\ref{prop:semi_rule} that the active sites of $\pi \cdot a$ are the same as those of $\pi$ (with $a$ split into two sites). 
This gives the second row of the production of $\Omega_{semi}$ (the label $(h+k+1-i,i)$ for $1 \leq i \leq k$ corresponding to $a$ being the $i$-th active site from the bottom). 

If $a=\pi_n +1$, $\pi \cdot a$ ends with an empty ascent, and hence has label $(h,k+1)$ again as in the proof of Proposition~\ref{prop:semi_rule}. 

Finally, if $a>\pi_n +1$ (which happens $h-1$ times), $\pi \cdot a$ ends with a non-empty ascent.
The discussion at the beginning of the proof implies that all sites of  $\pi \cdot a$ in $(\pi_n+1,a]$ are deactivated while all others remain active.
If $a$ is the $i$-th active site from the top, it follows that $\pi \cdot a$ has label $(i,k+1)$, hence giving all missing labels in the first row of the production of $\Omega_{semi}$. 
\end{proof}

Because the two families of semi-Baxter and of plane permutations grow according to the same succession rule, we obtain the following.  

\begin{corollary}\label{cor:equinum}
Semi-Baxter permutations and plane permutations are equinumerous. 
In other words, $SB_n$ is also the number of plane permutations of size $n$.
\end{corollary}

Note that the two generating trees for semi-Baxter and for plane permutations which are encoded by $\Omega_{semi}$ are of course isomorphic: 
this provides a size-preserving bijection between these two families. 
It is however not defined directly on the objects themselves, but only referring to the generating tree structure. 

\subsection{Another occurrence of semi-Baxter numbers}
\label{sec:inv}
In this section we provide another occurrence of semi-Baxter numbers that is not in terms of pattern-avoiding permutations, yet of pattern-avoiding inversion sequences. 
This occurrence appears as a conjecture in the work of M.~Martinez and C.~Savage on these families of objects~\cite{savage}. 

Recall that an \emph{inversion sequence} of size $n$ is an integer sequence $(e_1,e_2,\dots,e_n)$ satisfying $0\leq e_i< i$ for all $i\in\{1,2,\dots,n\}$. 
In \cite{savage} the authors introduce the notion of pattern avoidance in inversion sequences in quite general terms. 

Of interest to us here is only the set $\mathbf{I}_n(210,100)$ of inversion sequences avoiding the patterns $210$ and $100$: an inversion sequence $(e_1,\ldots,e_n)$ \emph{contains} the pattern $q$, with $q=q_1\ldots q_k$, if there exist $k$ indices $1\leq i_1<\ldots<i_k\leq n$ such that $e_{i_1}\ldots e_{i_k}$ is order-isomorphic to $q$; otherwise $(e_1,\ldots, e_n)$ is said to \emph{avoid} $q$. For instance, the sequence $(0,0,1,3,1,3,2,7,3)$ avoids both $210$ and $100$, while $(0,0,1,2,0,1,1)$ avoids $210$ but contains $100$.

The family $\cup_n\mathbf{I}_n(210,100)$ can moreover be characterized as follows.
A \emph{weak left-to-right maximum} of an inversion sequence $(e_1,e_2,\dots,e_n)$ is an entry $e_i$ satisfying $e_i \geq e_j$ for all $j\leq i$. 
Every inversion sequence $e$ can be decomposed in $e^{top}$, which is the (weakly increasing) sequence of weak left-to-right maxima of $e$, 
and $e^{bottom}$, which is the (possibly empty) sequence of the remaining entries of $e$. 
\begin{proposition}[\cite{savage}, Observation 10] \label{prop:charact_inv_seq}
An inversion sequence $e$ avoids $210$ and $100$ if and only if $e^{top}$ is weakly increasing and $e^{bottom}$ is strictly increasing.
\end{proposition}

The enumeration of inversion sequences avoiding $210$ and $100$ is solved in \cite{savage}, with a summation formula as reported in Proposition~\ref{prop:enum_inv_seq} below. 
Let top$(e)=\max\,(e^{top})$ and bottom$(e)=\max\, (e^{bottom})$. 
If $e^{bottom}$ is empty, the convention is to take $\text{bottom}(e)=-1$.

\begin{proposition}[\cite{savage}, Theorem 32] \label{prop:enum_inv_seq}
Let $Q_{n,a,b}$ be the number of inversion sequences $e\in\mathbf{I}_n(210,100)$ with top$(e)=a$ and
bottom$(e)=b$. Then 
\[Q_{n,a,b}=\sum_{i=-1}^{b-1}Q_{n-1,a,i}+\sum_{j=b+1}^a Q_{n-1,j,b},\]
with initial conditions $Q_{n,a,b}=0$, if $n\leq a$, and $Q_{n,a,-1}=\frac{n-a}{n}\binom{n-1+a}{a}$. Hence,
\begin{equation}\label{conjsav}
|\mathbf{I}_n(210,100)|=\sum_{a=0}^{n-1}\sum_{b=-1}^{a-1} Q_{n,a,b}=\frac{1}{n+1}\binom{2n}{n}+\sum_{a=0}^{n-1}\sum_{b=0}^{a-1} Q_{n,a,b}.
\end{equation}
\end{proposition}

We prove the following conjecture of~\cite[Section 2.27]{savage}. 
\begin{thm}\label{conjecture}
There are as many inversion sequences of size $n$ avoiding $210$ and $100$ as plane permutations of size $n$. In other words $|\mathbf{I}_n(210,100)| = SB_n$.
\end{thm}

\begin{proof}
We prove the statement by showing a growth for $\cup_n \mathbf{I}_n(210,100)$ which can be encoded by $\Omega_{semi}$.
Given an inversion sequence $e\in\mathbf{I}_n(210,100)$, we make it grow by adding a rightmost entry.

Let $a=\mbox{top}(e)$ and $b=\mbox{bottom}(e)$. 
From Proposition~\ref{prop:charact_inv_seq}, it follows that $f=(e_1,\dots,e_n,p)$ is an inversion sequence of size $n+1$ avoiding $210$ and $100$ 
if and only if $n\geq p>b$. 
Moreover, if $p\geq a$, then $f^{top}$ comprises $p$ in addition to the elements of $e^{top}$, and $f^{bottom}=e^{bottom}$;  
and if $b<p<a$, then $f^{top}=e^{top}$ and $f^{bottom}$ comprises $p$ in addition to the elements of $e^{bottom}$. 

Now, we assign to any $e\in\mathbf{I}_n(210,100)$ the label $(h,k)$, where $h=a-b$ and $k=n-a$. The sequence $e=(0)$ has label $(1,1)$, since $a=$ top$(e)=0$ and $b=$ bottom$(e)=-1$.
Let $e$ be an inversion sequence of $\mathbf{I}_n(210,100)$ with label $(h,k)$. 
The labels of the inversion sequences of $\mathbf{I}_{n+1}(210,100)$ produced adding a rightmost entry $p$ to $e$ are 
\begin{itemize}
 \item $(h+k,1), (h+k-1,2),\dots, (h+1,k)$ when $p = n, n-1,\dots, a+1$, 
 \item $(h,k+1)$ when $p=a$, 
 \item $(1,k+1), \dots, (h-1,k+1)$ when $p=a-1,\dots, b+1$, 
\end{itemize}
which concludes the proof that the growth of $\cup_n \mathbf{I}_n(210,100)$ by addition of a rightmost entry is encoded by $\Omega_{semi}$.
\end{proof}

\begin{remark}
In addition, in~\cite[Section 2.27]{savage} it is proved that the set $\mathbf{I}_n(210,100)$ has as many inversion sequences as the sets $\mathbf{I}_n(210,110)$, $\mathbf{I}_n(201,100)$, and $\mathbf{I}_n(201,101)$. 
Thus, our proof that the semi-Baxter numbers enumerate inversion sequences avoiding $210$ and $100$ also solves the enumeration problem of exactly four cases of~\cite[Table 2]{savage}.
\end{remark}
\subsection{Enumerative results}  \label{sec:resultsGF}

For $h,k\geq 1$, let $S_{h,k}(x)\equiv S_{h,k}$ denote the size generating function for semi-Baxter permutations having label $(h,k)$. 
The rule $\Omega_{semi}$ translates into a functional equation for the generating function $S(x;y,z)\equiv S(y,z)=\sum_{h,k\geq 1} S_{h,k} y^h z^k$. 

\begin{proposition} \label{prop:funcEqSemiBax}
The generating function $S(y,z)$ satisfies the following functional equation:
\begin{equation}\label{funeq1} S(y,z)= xyz+\frac{xyz}{1-y} \left( S(1,z) - S(y,z) \right) + \frac{xyz}{z-y} \left( S(y,z) - S(y,y) \right) .\end{equation}
\end{proposition}

\begin{proof}
Starting from the growth of semi-Baxter permutations according to $\Omega_{semi}$ we write:
\begin{align*}
S(y,z) &= xyz+x\sum_{h,k\geq 1} S_{h,k} \left( (y+y^2+\dots +y^h)z^{k+1}+(y^{h+k}z+y^{h+k-1}z^2+\dots +y^{h+1}z^k) \right) \\
&= xyz+x\sum_{h,k\geq 1} S_{h,k}\left( \frac{1-y^{h}}{1-y} y\, z^{k+1} + \frac{1-\left(\frac{y}{z}\right)^{k}}{1-\frac{y}{z}}y^{h+1} z^{k}\right) \\
&= xyz+ \frac{xyz}{1-y} \left ( S(1,z) - S(y,z) \right) + \frac{xyz}{z-y} \left ( S(y,z) - S(y,y) \right) \, .  \qedhere
\end{align*}
\end{proof}

From Proposition~\ref{prop:funcEqSemiBax}, a lot of information can be derived about the generating function $S(1,1)$ of semi-Baxter numbers, and about these numbers themselves. 
The results we obtain are stated below, but the proofs are postponed to Subsection~\ref{sec:proofs_GF}.  
A Maple worksheet recording the computations in these proofs is available from the authors' webpage\footnote{for instance at \url{http://user.math.uzh.ch/bouvel/publications/Semi-Baxter.mw}}. 

First, using the ``obstinate kernel method'' (used for instance in~\cite{BM} to enumerate Baxter permutations),  
we can give an expression for $S$. We let $\bar{a}$ denote $1/a$, and $\Omega_{\geq}[F(x;a)]$ denote the \emph{non-negative part} of $F$ in $a$, where $F$ is a formal power series in $x$ whose coefficients are Laurent polynomials in $a$. More precisely, if $F(x;a)=\sum_{n\geq0,i\in\mathbb{Z}}f(n,i)a^ix^n$, then
\[\Omega_{\geq}[F(x;a)]=\sum_{n\geq0}x^n\sum_{i\geq0}f(n,i)\,a^i.\]
\begin{thm}\label{thm:GFSemiBaxter}
Let $W(x;a)\equiv W$ be the unique formal power series in $x$ such that
\begin{equation*}
W=x\bar{a}(1+a)(W+1+a)(W+a).
\end{equation*}
The series solution $S(y,z)$ of eq.~\eqref{funeq1} satisfies 
\[S(1+a,1+a)={\Omega}_{\geq}\left[F(a,W)\right],\] 
where 
the function $F(a,W)$ is defined by
\begin{equation}\label{eqP}
\begin{array}{ll}F(a,W)=&(1+a)^2\,x+\left(\bar{a}^5+\bar{a}^4+2+2a\right)\,x\,W\\
\\
&+\left(-\bar{a}^5-\bar{a}^4+\bar{a}^3-\bar{a}^2-\bar{a}+1\right)\,x\,W^2+\,\left(\bar{a}^4-\bar{a}^2\right)\,x\,W^3.
\end{array}
\end{equation}\smallskip
\end{thm}

Note that in Theorem~\ref{thm:GFSemiBaxter}, $W$ and $F(a,W)$ are algebraic series in $x$ whose coefficients are Laurent polynomials in $a$ with rational coefficients. 
It follows, as in~\cite[page 6]{BM}, that $S(1+a,1+a)={\Omega}_{\geq}[F(a,W)]$ is D-finite\footnote{By definition, a multivariate generating function 
$F(\textbf{x})$, where $\textbf{x} = (x_1, \dots, x_k)$, is D-finite when it satisfies a system of linear partial differential equations, one for each 
$i = 1\dots k$, of the form $Q_{i,0}(\textbf{x})F(\textbf{x}) + Q_{i,1}(\textbf{x}) \frac{\partial}{\partial x_i} F(\textbf{x}) + 
Q_{i,2} \frac{\partial^2}{\partial x_i^2} F(\textbf{x}) + \dots + Q_{i,r_i} \frac{\partial^{r_i}}{\partial x_i^{r_i}} F(\textbf{x})= 0$, where the $Q_{i,j}$ are polynomials. 
As stated in~\cite[Theorem B.3]{Flaj}, D-finiteness is preserved by specialization of the variables.}, 
and hence also $S(1,1)$. 

Using the Lagrange Inversion, we can derive from Theorem~\ref{thm:GFSemiBaxter} an explicit but complicated expression for the coefficients of $S(1,1)$, 
which is reported in Corollary~\ref{cor:CoefSemiBaxter} in Subsection~\ref{sec:proofs_GF}. 
Surprisingly this complicated expression hides a very simple recurrence, which also appears as a conjecture in~\cite{bevan}. 

\begin{proposition}\label{prop:recSemiBaxter}
The numbers $SB_n$ are recursively characterized by $SB_0=0$, $SB_1=1$ and for $n\geq2$
\begin{equation}\label{recurrenceSB}
SB_n=\frac{11n^2+11n-6}{(n+4)(n+3)}SB_{n-1}+\frac{(n-3)(n-2)}{(n+4)(n+3)}SB_{n-2}.
\end{equation}
\end{proposition}

From the recurrence of Proposition~\ref{prop:recSemiBaxter}, we can in turn prove closed formulas for semi-Baxter numbers, 
which have been conjectured in~\cite{BevanPrivate}. 
These are much simpler than the one given in Corollary~\ref{cor:CoefSemiBaxter} by the Lagrange inversion, 
and also very much alike the summation formula for Baxter numbers (which we recall in Subsection~\ref{sec:Bax_known}). 

\begin{thm}\label{thm:nice_formulas_SB_n}
For any $n\geq 2$, the number $SB_n$ of semi-Baxter permutations of size $n$ satisfies 
\begin{align*}
SB_n & = \frac{24}{(n-1) n^2 (n+1) (n+2)} \sum_{j=0}^n \binom{n}{j+2} \binom{n+2}{j} \binom{n+j+2}{j+1} \\ 
& = \frac{24}{(n-1) n^2 (n+1) (n+2)} \sum_{j=0}^n \binom{n}{j+2} \binom{n+1}{j} \binom{n+j+2}{j+3}\\
& = \frac{24}{(n-1) n^2 (n+1) (n+2)} \sum_{j=0}^n \binom{n+1}{j+3} \binom{n+2}{j+1} \binom{n+j+3}{j}.
\end{align*}
\end{thm}

There is actually a fourth formula that has been conjectured in~\cite{BevanPrivate}, namely 
\[
SB_n = \frac{24}{(n-1) n (n+1)^2 (n+2)} \sum_{j=0}^n \binom{n+1}{j} \binom{n+1}{j+3} \binom{n+j+2}{j+2}.
\]
Taking the multiplicative factors inside the sums, it is easy to see (for instance going back to the definition of binomial coefficients as quotients of factorials) 
that it is term by term equal to the second formula of Theorem~\ref{thm:nice_formulas_SB_n}. 

As indicated in Subsection~\ref{sec:intro_semiBax}, in addition to the formulas reported in Theorem~\ref{thm:nice_formulas_SB_n} above, 
two conjectural formulas for $SB_n$ have been proposed in the literature, in different contexts. 

The first one has been shown in Subsection~\ref{sec:inv} (eq.~\eqref{conjsav}). This formula was proved in~\cite{savage} and its validity for semi-Baxter numbers follows by Theorem~\ref{conjecture}. 
The second formula is attributed to M. Van Hoeij and reported by D. Bevan in~\cite{bevan}. 
This second conjecture is an explicit formula for semi-Baxter numbers that involves the numbers $a_n=\sum_{j=0}^n\binom{n}{j}^2\binom{n+j}{j}$ (sequence \text{A005258} on~\cite{OEIS}). 
We will prove in Subsection~\ref{sec:proofs_GF} the validity of this conjecture by using the recursive formula for semi-Baxter numbers (Proposition~\ref{prop:recSemiBaxter}).

\begin{proposition}[\cite{bevan}, Conjecture 13.2]\label{prop:conj}
For $n\geq2$,$$SB_n=\frac{24}{5}\,\frac{(5n^3-5n+6)a_{n+1}-(5n^2+15n+18)a_n}{(n-1)n^2(n+2)^2(n+3)^2(n+4)}$$
\end{proposition}

\begin{remark}
With Corollary~\ref{cor:CoefSemiBaxter}, Theorem~\ref{thm:nice_formulas_SB_n} and Proposition~\ref{prop:conj}, 
we get five expressions for the $n$th semi-Baxter number as a sum over $j$. 
Note that although the sums are equal, the corresponding summands in each sum are not. 
Therefore, Corollary~\ref{cor:CoefSemiBaxter}, Theorem~\ref{thm:nice_formulas_SB_n} and Proposition~\ref{prop:conj} give five essentially different ways of expressing the semi-Baxter numbers.
Note however that we are not aware of any combinatorial interpretation of the summation index $j$, for any of them.
\end{remark}

From the formula of Proposition~\ref{prop:conj}, we can derive the dominant asymptotics of $SB_n$. 
\begin{corollary}\label{cor:asym_SB_n}
Let $\lambda = \frac{1}{2}(\sqrt{5}-1)$. It holds that 
\[SB_n \sim  A\,\frac{\mu^{n}}{n^{6}},\] 
where $A=\frac{12}{\pi}\,5^{-1/4} \lambda^{-15/2}\approx94.34$ and $\mu=\lambda^{-5}= (11+5\sqrt{5})/2$.
\end{corollary}
\subsection{Enumerative results: proofs}\label{sec:proofs_GF}

Recall that $S(y,z)$ denotes the multivariate generating function of semi-Baxter permutations. 
In Theorem~\ref{thm:GFSemiBaxter}, we have given an expression for $S(1+a,1+a)$, which we now prove. 

\begin{proof}[Proof of Theorem~\ref{thm:GFSemiBaxter}] 
The linear functional equation of eq.~\eqref{funeq1} has two catalytic variables, $y$ and $z$. 
 To solve eq.~\eqref{funeq1} it is convenient to set $y=1+a$ and collect all the terms having $S(1+a,z)$ in them, obtaining the {\em kernel form} of the equation:
\begin{equation}
\label{eq1az}
K(a,z) S(1+a,z)= xz(1+a)-\frac{xz(1+a)}{a}  S(1,z) - \frac{xz(1+a)}{z-1-a} S(1+a,1+a)  ,
\end{equation} 
where the kernel is 
\[K(a,z)=1-\frac{xz(1+a)}{a}- \frac{xz(1+a)}{z-1-a}.\] 
For brevity, we refer to the right-hand side of eq.~\eqref{eq1az} as $R(x,a,z,S(1,z),S(1+a,1+a))$.

The kernel is quadratic in $z$. Denoting $Z_{+}(a)$ and $Z_{\_}(a)$ the solutions of $K(a,z)=0$ with respect to $z$, and $Q=\sqrt{a^2-2ax-6a^2x+x^2+2ax^2+a^2x^2-4a^3x}$, we have 
\begin{align*}
Z_{+}(a) & =\frac{1}{2}\frac{a+x+ax-Q}{x(1+a)}=(1+a)+(1+a)^2x+\frac{(1+a)^3(1+2a)}{a}x^2+O(x^3),\\
Z_{\_}(a) & =\frac{1}{2}\frac{a+x+ax+Q}{x(1+a)}=\frac{a}{(1+a)x}-a-(1+a)^2x-\frac{(1+a)^3(1+2a)}{a}x^2+O(x^3).
\end{align*}

Both $Z_{+}$ and $Z_{\_}$ are Laurent series in $x$ whose coefficients are Laurent polynomials in $a$.
However, only the kernel root $Z_{+}$ is a formal power 
series in $x$ whose coefficients are Laurent polynomials in $a$. 
So, setting $z=Z_{+}$, the function 
$S(1+a,z)$ is a formal power series in $x$ whose coefficients are Laurent polynomials in $a$, and the right-hand side of eq.~\eqref{eq1az} is equal to zero, \emph{i.e.} 
$R(x,a,Z_{+},S(1,Z_{+}),S(1+a,1+a))=0$. 
Note in addition that the coefficients of $Z_{+}$ are multiples of $(1+a)$.

At this point we follow the usual kernel method (see for instance~\cite{BM}) and attempt to eliminate the term $S(1,Z_+)$ by 
exploiting transformations that leave the kernel, $K(a,z)$, unchanged. Examining the kernel shows that the 
transformations $$\Phi:(a,z)\rightarrow\left(\frac{z-1-a}{1+a},z\right)\;\;\;\mbox{ and 
}\;\;\;\Psi:(a,z)\rightarrow\left(a,\frac{z+za-1-a}{z-1-a}\right)$$ leave the kernel unchanged and generate a group 
of order $10$.

Among all the elements of this group we consider the following pairs $(f_1(a,z),f_2(a,z))$: 
$$\left[a,z\right]\mathop{\longleftrightarrow}_{\Phi}\left[\frac{z- 1- a}{1+a},z\right]\mathop{\longleftrightarrow}_{\Psi}\left[\frac{z- 1- a}{1+a},\frac{z-1}{a}\right]\mathop{\longleftrightarrow}_{\Phi}\left[\frac{z- 1- a}{az},\frac{z-1}{a}\right]\mathop{\longleftrightarrow}_{\Psi}
\left[\frac{z- 1- a}{az},\frac{1+a}{a}\right].$$
These have been chosen since, for each of them, $f_1(a,Z_{+})$ and $f_2(a,Z_{+})$ are formal power series in $x$ with Laurent polynomial coefficients in $a$. 
Consequently, they share the property that $S(1+f_1(a,Z_{+}),f_2(a,Z_{+}))$ are formal power series in $x$. 
It follows that, substituting each of these pairs for $(a,z)$ in eq.~\eqref{eq1az}, 
we obtain a system of five equations, whose left-hand sides are all $0$, and with six unknowns: 

\[
\begin{cases}
 0=R(x,a,Z_{+},S(1,Z_{+}),S(1+a,1+a))\\
 \\
 0=R\left(x,\frac{Z_{+}- 1- a}{1+a},Z_{+},S(1,Z_{+}),S(1+\frac{Z_{+}- 1- a}{1+a},1+\frac{Z_{+}- 1- a}{1+a})\right)\\
 \\
   0=R\left(x,\frac{Z_{+}-1}{a},\frac{Z_{+}- 1- a}{1+a},S(1,\frac{Z_{+}-1}{a}),S(1+\frac{Z_{+}- 1- a}{1+a},1+\frac{Z_{+}- 1- a}{1+a})\right) \\
 \\
   0=R\left(x,\frac{Z_{+}- 1- a}{aZ_{+}},\frac{Z_{+}-1}{a},S(1,\frac{Z_{+}-1}{a}),S(1+\frac{Z_{+}- 1- a}{aZ_{+}},1+\frac{Z_{+}- 1- a}{aZ_{+}})\right)\\
\\
   0=R\left(x,\frac{Z_{+}- 1- a}{aZ_{+}},\frac{1+a}{a},S(1,\frac{1+a}{a}),S(1+\frac{Z_{+}- 1- a}{aZ_{+}},1+\frac{Z_{+}- 1- a}{aZ_{+}})\right).
\end{cases}
\]

Eliminating all unknowns except $S(1+a, 1+a)$ and $S(1,1+\bar{a})$, this system reduces (after some work) to the following equation:
\begin{equation}
\label{final}
S(1+a, 1+a)+\frac{(1+a)^2x}{a^4}S\left(1,1+\bar{a}\right)+P(a,Z_{+})=0, 
\end{equation}
where $P(a,z)=(-z+1+a)(-za^4+z^2a^4-za^3+z^2a^3-z^3a^2-2a^2+z^2a^2+za^2-4a+5az-3az^2+z^3a+3z-z^2-2)/({z}{a}^4(z-1))$.
Note that the coefficient of $S(1,1+\bar{a})$ in eq.~\eqref{final} results to be equal to $(1+a)^2x\bar{a}^4$ only after setting $z=Z_+$ and simplifying the expression obtained.

Now, the form of eq.~\eqref{final} allows us to separate its terms according to the power of $a$: 
\begin{itemize}
 \item $S(1+a,1+a)$ is a power series in $x$ with polynomial coefficients in $a$ whose lowest power of $a$ is $0$, 
 \item $S(1,1+\bar{a})$ is a power series in $x$ with polynomial coefficients in $\bar{a}$ whose highest power of $a$ is 
$0$; consequently, 
we obtain that $\frac{(1+a)^2x}{a^4}S(1,1+\bar{a})$ is a power series 
in $x$ with polynomial coefficients in $\bar{a}$ whose highest power of $a$ is $-2$.
\end{itemize}
Hence when we expand the series $-P(a,Z_+)$ as a power series in $x$, the non-negative powers of $a$ in the 
coefficients must be equal to those of $S(1+a,1+a)$, while the negative powers of $a$ come from 
$\frac{(1+a)^2x}{a^4}S(1,1+\bar{a})$.

Then, in order to have a better expression for $P(a,z)$, we perform a further substitution setting $z=w+1+a$. 
More precisely, let $W\equiv W(x;a)$ be the power series in $x$ defined by $W=Z_{+}-(1+a)$. 
Since $K(a,W+1+a)=0$, the function $W$ is recursively defined by 
\begin{equation}
\label{Wexpr}
W=x\bar{a}(1+a)(W+1+a)(W+a),
\end{equation}
as claimed. 
Moreover, we have the following expression for $F(a,W) := -P(a,Z_{+})$:
\begin{equation*}
\begin{array}{ll}F(a,W)  = - P(a,W+1+a)=&\displaystyle\,(1+a)^2\,x+\left(\frac{1}{a^5}+\frac{1}{a^4}+2+2a\right)\,x\,W\\[2ex]
&\displaystyle+\left(-\frac{1}{a^5}-\frac{1}{a^4}+\frac{1}{a^3}-\frac{1}{a^2}-\frac{1}{a}+1\right)\,x\,W^2\\[2ex]
&\displaystyle+\left(\frac{1}{a^4}-\frac{1}{a^2}\right)\,x\,W^3,
\end{array}
\end{equation*}
in which the denominator of $- P(a,W+1+a)$ is eliminated by substituting in it a factor $W$ for the right-hand side of eq.~\eqref{Wexpr}.
\end{proof}

From the expression of $S(1+a,1+a)$ obtained above, the Lagrange inversion allows us to derive an explicit expression for the semi-Baxter numbers, as shown below in Corollary~\ref{cor:CoefSemiBaxter}. 
From it, we next obtain the simple recurrence of Proposition~\ref{prop:recSemiBaxter}, the conjectured simpler formulas for $SB_n$ given in Theorem~\ref{thm:nice_formulas_SB_n},
and the asymptotic estimate of $SB_n$ stated in Corollary~\ref{cor:asym_SB_n}.

\begin{corollary}\label{cor:CoefSemiBaxter}
The number $SB_n$ of semi-Baxter permutations of size $n$ satisfies, for all $n\geq2$:
\begin{eqnarray*}
{SB}_n&\hspace{-3mm}=\frac{1}{n-1}\sum_{j=0}^{n}\binom{n-1}{j}\biggr[\binom{n-1}{j+1}\left[\binom{n+j+1}{j+5}+2\binom{n+j+1}{j}\right]+2\binom{n-1}{j+2}\left[-\binom{n+j+2}{j+5}+\binom{n+j+1}{j+3}\right.\\[-.5ex]
&\left.\qquad\quad\quad-\binom{n+j+2}{j+2}+\binom{n+j+1}{j}\right]+3\binom{n-1}{j+3}\left[\binom{n+j+2}{j+4}-\binom{n+j+2}{j+2}\right]\biggr].
\end{eqnarray*}
\end{corollary}

\begin{proof}
The $n$th semi-Baxter number, $SB_n$, is the coefficient of $x^n$ in $S(1,1)$, which we denote as usual $[x^n]S(1,1)$. 
Notice that this number is also the coefficient $[a^0x^n]S(1+a,1+a)$, and so by Theorem~\ref{thm:GFSemiBaxter} it is the coefficient of $a^0x^n$ in $F(a,W) = - P(a,W+1+a)$,  
namely 
\[\begin{array}{ll}SB_n=\displaystyle[a^0x^{n-1}]&\hspace{-3mm}\displaystyle\left((1+a)^2+\left(\frac{1}{a^5}+\frac{1}{a^4}+2+2a\right)\,W+\left(-\frac{1}{
a^5}-\frac{1}{a^4}+\frac{1}{a^3}-\frac{1}{a^2}-\frac{1}{a}+1\right)\,W^2\right.\\[3ex]
&\hspace{-3mm}\displaystyle+\left.\left(\frac{1}{a^4}-\frac{1}{a^2}\right)\,W^3\right).\end{array}\] This expression can be evaluated from $[a^sx^k]W^i$, for $i=1,2,3$. Precisely,
\[\begin{array}{ll}
SB_n=&\hspace{-3mm}[a^5x^{n-1}]W+[a^4x^{n-1}]W+2[a^0x^{n-1}]W+2[a^{-1}x^{n-1}]W-[a^5x^{n-1}]W^2-[a^4x^{n-1}]W^2\\
\\
&\hspace{-3mm}+[a^3x^{n-1}]W^2-[a^2x^{n-1}]W^2-[a^1x^{n-1}]W^2+\left[a^0x^{n-1}\right]W^2+[a^4x^{n-1}]W^3-[a^2x^{n-1}]W^3.\end{array}\]
The Lagrange inversion and eq.~\eqref{Wexpr} then prove that
$$[a^sx^k]W^i=\frac{i}{k}\sum_{j=0}^{k-i}\binom{k}{j}\binom{k}{j+i}\binom{k+j+i}{j+s},\mbox{ for }i=1,2,3.$$
We can then substitute this into the above expression for $SB_n$ and, for $n\geq2$, obtain the announced explicit formula for the semi-Baxter coefficients $SB_n$ setting $SB_n=\sum_{j=0}^{n-1} 
F_{SB}(n,j)$, where
\belowdisplayskip=-10pt
\begin{eqnarray}
\notag
F_{SB}(n,j)\hspace{-6mm}&&\displaystyle=\frac{1}{n-1}\binom{n-1}{j}\Biggr[\binom{n-1}{j+1}\left[\binom{n+j+1}{j+5}+2\binom{n+j+1}{j}\right]\\[1ex]\label{eq:summand}
&&\quad\displaystyle +2\,\binom{n-1}{j+2}\left[-\binom{n+j+2}{j+5}+\binom{n+j+1}{j+3}-\binom{n+j+2}{j+2}+\binom{n+j+1}{j}\right]\\[1ex]\notag
&&\quad\displaystyle+3\,\binom{n-1}{j+3}\left[\binom{n+j+2}{j+4}-\binom{n+j+2}{j+2}\right]\Biggr].
\end{eqnarray}
\end{proof}

\begin{proof}[Proof of Proposition~\ref{prop:recSemiBaxter}]
From Corollary~\ref{cor:CoefSemiBaxter}, we can write $SB_n=\sum_{j=0}^{n-1} F_{SB}(n,j)$, where the summand 
$F_{SB}(n,j)$ given by eq.~\eqref{eq:summand} is hypergeometric, and we prove the announced recurrence using \emph{creative telescoping}~\cite{zeilberger}. The Maple 
package {\tt SumTools[Hypergeometric][Zeilberger]} implements this approach: using $F_{SB}(n,j)$ as input, it yields
\begin{multline}\label{zeil1}
(n+5)(n+6)\cdot F_{SB}(n+2,j)
- (11n^2+55n+60)\cdot F_{SB}(n+1,j)
- n(n-1)\cdot F_{SB}(n,j)\\
= G_{SB}(n, j + 1)-G_{SB}(n, j),
\end{multline}
where $G_{SB}(n,j)$ is known as the certificate. It has the additional property that $G_{SB}(n,j) / F_{SB}(n,j)$ is a 
rational function of $n$ and $j$. The expression $G_{SB}(n,j)$ is quite cumbersome and we do not report it here --- it 
can be readily reconstructed using \texttt{Zeilberger} as done in the Maple worksheet associated with our paper. 

To complete the proof of the recurrence it is sufficient to sum both sides of eq.~\eqref{zeil1} over $j$, $j$ ranging from $0$ to $n+1$. 
Since the coefficients on the left-hand side of eq.~\eqref{zeil1} are independent of $j$, summing it over $j$ gives
\begin{multline}
(n+5)(n+6) \cdot SB_{n+2}
- (11n^2+55n+60)\cdot SB_{n+1}
- n(n-1)\cdot SB_n \\
- (11n^2+55n+60) \cdot F_{SB}(n+1,n+1)
- n(n-1)\cdot (F_{SB}(n,n)+F_{SB}(n,n+1)).
\end{multline}
Summing the right-hand side over $j$ gives a telescoping series, and simplifies as $G_{SB}(n, n + 2)-G_{SB}(n, 0)$.
From the explicit expression of $F_{SB}(n,j)$ and $G_{SB}(n,j)$, it is elementary to check that 
\[
 F_{SB}(n+1,n+1) = F_{SB}(n,n) = F_{SB}(n,n+1) = G_{SB}(n, n + 2) = G_{SB}(n, 0) =0.
\]
Summing eq.~\eqref{zeil1} therefore gives 
\[
 (n+5)(n+6) \cdot SB_{n+2}
- (11n^2+55n+60)\cdot SB_{n+1}
- n(n-1)\cdot SB_n =0.
\]
Shifting $n \mapsto n-2$ and rearranging finally gives the recurrence of Proposition~\ref{prop:recSemiBaxter}.
\end{proof}

\begin{proof}[Proof of Theorem~\ref{thm:nice_formulas_SB_n}]
For each of the summation formulas given in Theorem~\ref{thm:nice_formulas_SB_n}, we apply the method of creative telescoping, as in the proof of Proposition~\ref{prop:recSemiBaxter}.
In all three cases, this produces a recurrence satisfied by these numbers, and every time we find exactly the recurrence given in Proposition~\ref{prop:recSemiBaxter}. 
Checking that the initial terms of the sequences coincide completes the proof. 
\end{proof}

\begin{proof}[Proof of Proposition~\ref{prop:conj}]
For the sake of brevity we write $A(n)=5n^3-5n+6$ and $B(n)=5n^2+15n+18$ so that the statement becomes \begin{equation}\label{pr:conj}
SB_n=\frac{24(A(n)\,a_{n+1}-B(n)\,a_n)}{5(n-1)n^2(n+2)^2(n+3)^2(n+4)}.
\end{equation}
The validity of eq.~\eqref{pr:conj} is proved by induction on $n$ using Proposition~\ref{prop:recSemiBaxter} and the following recurrence satisfied by the numbers $a_n$, for $n\geq1$: 
\begin{equation}\label{rec:aperi}
a_{n+1}=\frac{11n^2+11n+3}{(n+1)^2}\,a_n+\frac{n^2}{(n+1)^2}\,a_{n-1}, \mbox{ with }a_0=1,\mbox{ and }a_1=3.
\end{equation}

For $n=2,3$, it holds that $SB_2=({A(2)a_3-B(2)a_2})/{2000}=({36\cdot147-68\cdot19})/{2000}=2$ and $SB_3=({A(3)a_4-B(3)a_3})/{23625}=({126\cdot1251-108\cdot147})/{23625}=6$. 

Then, suppose that eq.~\eqref{pr:conj} is valid for $n-1$ and $n-2$. 
In order to prove it for $n$, consider the recursive formula of eq.~\eqref{recurrenceSB} and substitute in it $SB_{n-1}$ and $SB_{n-2}$ by using eq.~\eqref{pr:conj}. 
Now, after some work of manipulation and by using eq.~\eqref{rec:aperi} we can write $SB_n$ as in eq.~\eqref{pr:conj}.
\end{proof}

\begin{proof}[Proof of Corollary~\ref{cor:asym_SB_n}]
Applying the main theorem of~\cite{McIntosh}, it follows immediately that 
\[
 a_n \sim \frac{\mu^{n+1/2}}{2 \pi \lambda n \sqrt{\nu}} \textrm{ \quad for } \lambda = \frac{\sqrt{5}-1}{2}, \mu= \frac{11+5\sqrt{5}}{2} \textrm{ and } \nu=\frac{2\sqrt{5}}{3-\sqrt{5}}.
\]
Plugging this expression in the relation 
\[
SB_n=\frac{24}{5(n-1)n^2(n+2)^2(n+3)^2(n+4)}\,\big((5n^3-5n+6)a_{n+1}-(5n^2+15n+18)a_n\big),   
\]
we see that only the first of these two terms contributes to the asymptotic behavior of $SB_n$, 
and more precisely that 
\[
SB_n \sim A \mu^ n n^{-6} \textrm{ \quad for } \mu \textrm{ as above and } A=\frac{24 \mu^{3/2}}{2 \pi \lambda \sqrt{\nu}}.
\]
The claimed statement then follows noticing that $\mu=\lambda^{-5}$ and $\nu = \frac{\sqrt{5}}{\lambda^2}$.
\end{proof}

\section{Baxter numbers}\label{sec:Bax}

This section starts with an overview of some known results about Baxter numbers. 
We believe it helps understanding the relations, similarities and differences between this well-known sequence and the two main sequences studied in our work 
(semi-Baxter numbers in Section~\ref{sec:semi} and strong-Baxter numbers in Section~\ref{sec:strong}). 
Next, studying two families of restricted semi-Baxter permutations enumerated by Baxter numbers, 
we show that $\Omega_{semi}$ generalizes two known succession rules for Baxter numbers.

\subsection{Baxter numbers and restricted permutations}\label{sec:Bax_known}

\emph{Baxter permutations} (see~\cite{Gire} among others) are usually defined as permutations avoiding the two vincular patterns $\commonpat$ and $\baxpat$. 
Denoting $B_n$ the number of Baxter permutations of size $n$, the sequence $(B_n)$ is known as the sequence of \emph{Baxter numbers}. 
It is identified as sequence \text{A001181} in~\cite{OEIS} and its first terms are $1,2,6,22,92,422,2074,10754, 58202, 326240, 1882960, 11140560,\dots$. 
Since~\cite{CGHK78}, an explicit formula for $B_n$ has been known: 
\begin{equation*}
\textrm{for all } n\geq1,\ B_{n}=\frac{2}{n(n+1)^2}\sum_{j=1}^{n}\binom{n+1}{j-1}\binom{n+1}{j}\binom{n+1}{j+1}. 
\end{equation*}
In~\cite{BM}, M. Bousquet-M\'elou investigates further properties of Baxter numbers. 
The above formula can also be found in~\cite[Theorem 1]{BM}. 
Moreover, using the succession rule reviewed in Proposition~\ref{prop:growthBaxterPerm} below, 
\cite{BM} characterizes the generating function of Baxter numbers as the solution of a bivariate functional equation. 
It is then solved with the obstinate kernel method, implying that the generating function for Baxter numbers is D-finite~\cite[Theorem 4]{BM}. 
Although technical details differ, it is the same approach than the one we used in Section~\ref{sec:semi}. 
In the light of our recurrence for semi-Baxter numbers (see Proposition~\ref{prop:recSemiBaxter}), 
it is also interesting to note that Baxter numbers satisfy a similar recurrence, reported by R. L. Ollerton in~\cite{OEIS}, namely 
\begin{equation*}B_0=0, \quad B_1=1, \quad \text{ and for } n\geq2, B_n=\frac{7n^2+7n-2}{(n+3)(n+2)}B_{n-1}+\frac{8(n-2)(n-1)}{(n+3)(n+2)}B_{n-2}.\end{equation*}

In addition to Baxter permutations, several combinatorial families are enumerated by Baxter numbers. 
See for instance~\cite{ffno} which collects some of them and provides links between them. 
We will be specifically interested in a second family of restricted permutations which is also enumerated by Baxter numbers, 
namely the \emph{twisted Baxter permutations}, defined by the avoidance of $\commonpat$ and $\twistedpat$~\cite{Rea05,West}. 

\subsection{Succession rules for Baxter and twisted Baxter permutations}

It is clear from their definition in terms of pattern-avoidance that the families of Baxter and twisted Baxter permutations are 
subsets of the family of semi-Baxter permutations. 
Therefore, the growth of semi-Baxter permutations provided in Subsection~\ref{sec:GenTreeSemiBax} can be restricted to each of these families, 
producing a succession rule for Baxter numbers. 
In the following, we present these two restrictions, which happen to be (variants of) well-known succession rules for Baxter numbers. 
This reinforces our conviction that the generalization of Baxter numbers to semi-Baxter numbers is natural. 

\medskip

Let us first consider Baxter permutations. 
To that effect, recall that a LTR (left-to-right) maximum of a permutation $\pi$ is an element $\pi_i$ such that $\pi_i > \pi_j$ for all $j<i$. 
Similarly, a RTL maximum (resp. RTL minimum) of $\pi$ is an element $\pi_i$ such that $\pi_i > \pi_j$ (resp. $\pi_i < \pi_j$) for all $j>i$. 
Following~\cite[Section 2.1]{BM} we can make Baxter permutations grow by adding new maximal elements to them, 
which may be inserted either immediately before a LTR maximum or immediately after a RTL maximum. 
Giving to any Baxter permutation the label $(h,k)$ where $h$ (resp. $k$) is the number of its RTL (resp. LTR) maxima, 
this gives the most classical succession rule for Baxter numbers. 
\begin{proposition}[\cite{BM}, Lemma 2] \label{prop:growthBaxterPerm}
The growth of Baxter permutations by insertion of a maximal element is encoded by the rule 
$$\Omega_{Bax}=\left\{\begin{array}{ll}
(1,1)\\
(h,k) \rightsquigarrow \hspace{-3mm}& (1,k+1), \dots , (h,k+1)\\
& (h+1,1), \dots , (h+1,k), \end{array}\right.$$
where $h$ (resp. $k$) is the number of RTL (resp. LTR) maxima.
\end{proposition}

But note that Baxter permutations are invariant under the $8$ symmetries of the square. 
Consequently, up to a $90^\circ$ rotation, inserting a new maximum element in a Baxter permutation can be easily regarded as inserting a new element on the right of a Baxter permutation 
(as we did for semi-Baxter permutations). 
Those are then inserted immediately below a RTL minimum or immediately above a RTL maximum. 
Note that, in a semi-Baxter permutation, these are always active sites, so 
the generating tree associated with $\Omega_{Bax}$ is a subtree of the generating tree associated with $\Omega_{semi}$.
Through the rotation, the interpretation of the label $(h,k)$ of a Baxter permutation is modified as follows: $h$ (resp. $k$) is the number of its RTL minima (resp. RTL maxima), 
that is to say of active sites below (resp. above) the last element of the permutation. 
As expected, this coincides with the interpretation of labels in the growth of semi-Baxter permutations according to $\Omega_{semi}$.

\medskip

Turning to twisted Baxter permutations, specializing the growth of semi-Baxter permutations, we obtain the following. 
\begin{proposition}\label{prop:twist_rule}
A generating tree for twisted Baxter permutations can be obtained by insertions on the right, 
and it is isomorphic to the tree generated by the following succession rule: 
$$\Omega_{TBax}=\left\{\begin{array}{ll}
(1,1)\\
(h,k) \rightsquigarrow \hspace{-3mm}& (1,k), \dots ,(h-1,k), (h,k+1)\\
& (h+k,1), \dots , (h+1,k). \end{array}\right.$$
\end{proposition}

\begin{proof}
As in the proof of Proposition~\ref{prop:semi_rule}, we let twisted Baxter permutations grow by performing local expansions on the right, 
as illustrated in Figure~\ref{fig:twisted_growth}. 
(This is possible since removing the last element in a twisted Baxter permutation produces a twisted Baxter permutation.) 

Let $\pi$ be a twisted Baxter permutation of size $n$. By definition an active site of $\pi$ is an element $a$ such that $\pi\cdot a$ avoids the two forbidden patterns. 
Then, we assign to $\pi$ a label $(h,k)$, where $h$ (resp. $k$) is the number of active sites smaller than or equal to (resp. greater than) $\pi_n$. 
As in the proof of Proposition~\ref{prop:semi_rule}, the permutation $1$ has label $(1,1)$ and now we describe the labels of the permutations $\pi\cdot a$ when $a$ runs over all the active sites of $\pi$.

If $a<\pi_n$, then $\pi\cdot a$ ends with a non-empty descent and, as in the proof of Proposition~\ref{prop:semi_rule}, 
all sites of $\pi$ in the range $(a+1,\pi_n+1]$ become non-active in $\pi\cdot a$ (due to the avoidance of $\commonpat$). 
Moreover, due to the avoidance of $\twistedpat$, the site immediately above $a$ in $\pi\cdot a$ also becomes non-active. 
All other active sites of $\pi$ remain active in $\pi\cdot a$, hence giving the labels $(i,k)$, for $1\leq i<h$, in the productions of $\Omega_{TBax}$ 
($(i,k)$ corresponds to the case where $a$ is the $i$th active site from the bottom). 

If $a=\pi_n$, no sites of $\pi$ become non-active, giving the label $(h,k+1)$.

If $a>\pi_n$, then $\pi\cdot a$ ends with an ascent and no site of $\pi$ become non-active.
Hence, we obtain the missing labels in the production of $\Omega_{TBax}$: $(h+k+1-i,i)$, for $1\leq i\leq k$ 
(the label $(h+k+1-i,i)$ corresponds to $a$ being the $i$th active site from the top).
\end{proof}

\begin{figure}[ht]
\centering
\includegraphics[width=0.6\textwidth]{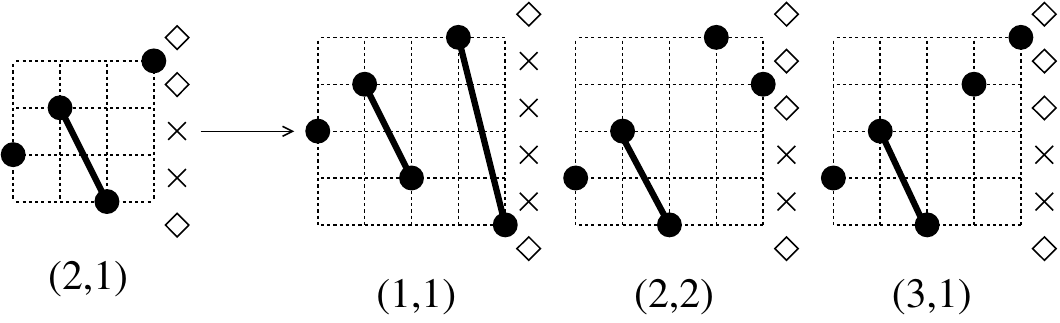}
\caption{The growth of twisted Baxter permutations (with notation as in Figure~\ref{fig:semi}).}
\label{fig:twisted_growth}
\end{figure}

We remark that although $\Omega_{TBax}$ is not precisely the succession rule presented in~\cite{Twis} for twisted Baxter permutations, 
it is an obvious variant of it: indeed, starting from the rule of~\cite{Twis}, it is enough to replace every label $(q,r)$ by $(r+1,q-1)$ to recover $\Omega_{TBax}$. 

It follows immediately from the proof of Proposition~\ref{prop:twist_rule} that $\Omega_{TBax}$ is a specialization of $\Omega_{semi}$. 
With $\Omega_{Bax}$, we therefore obtain two such specializations. 
In addition, we can observe that the productions of $\Omega_{TBax}$ on second line are the same as in $\Omega_{semi}$, 
whereas the productions on the first line of $\Omega_{Bax}$ are the same as $\Omega_{semi}$. 
This means that the restrictions imposed by these two specializations are ``independent''. 
We will combine them in Section~\ref{sec:strong}, obtaining a succession rule which consists of the first line of $\Omega_{TBax}$ and the second line of $\Omega_{Bax}$. 

\section{Strong-Baxter numbers} \label{sec:strong}
While Section~\ref{sec:semi} was studying a sequence larger than the Baxter numbers (with a family of permutations containing both the Baxter and twisted Baxter permutations), 
we now turn to a sequence smaller than the Baxter numbers (associated with a family of permutations included in both families of Baxter and twisted Baxter permutations). 
We present a succession rule for this sequence, and properties of its generating function. 

\subsection{Strong-Baxter numbers, strong-Baxter permutations, and their succession rule}

\begin{definition}
A \emph{strong-Baxter permutation} is a permutation that avoids all three vincular patterns $\commonpat$, $\baxpat$ and $\twistedpat$.
\end{definition}

\begin{definition}
The sequence of \emph{strong-Baxter numbers} is the sequence that enumerates strong-Baxter permutations. 
\end{definition}

We have added the sequence enumerating strong-Baxter permutations to the OEIS, where it is now registered as \cite[A281784]{OEIS}. It starts with:
\[1,2,6,21,82,346,1547,, 7236, 35090, 175268, 897273, 4690392, 24961300, \ldots\]
The pattern-avoidance definition makes it clear that the family of strong-Baxter permutations is the intersection of the two families of Baxter and twisted Baxter permutations. 
In that sense, these permutations ``satisfy two Baxter conditions'', hence the name strong-Baxter. 

\medskip

A succession rule for strong-Baxter numbers is given by the following proposition. 

\begin{proposition} 
A generating tree for strong-Baxter permutations can be obtained by insertions on the right, 
and it is isomorphic to the tree generated by the following succession rule: 
$$\Omega_{strong}=\left\{\begin{array}{ll}
(1,1)\\
(h,k) \rightsquigarrow \hspace{-3mm}& (1,k), \dots ,(h-1,k), (h,k+1)\\
& (h+1,1), \dots , (h+1,k). \end{array}\right.$$
\end{proposition}

\begin{figure}[ht]
\centering
\includegraphics[width=0.8\textwidth]{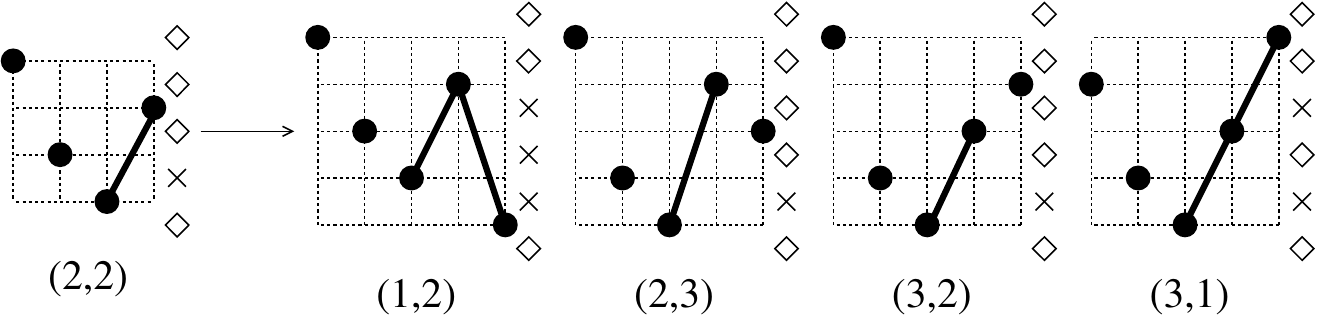}
\caption{The growth of strong-Baxter permutations (with notation as in Figure~\ref{fig:semi}).}
\label{fig:strong}
\end{figure}

\begin{proof}
As in the proof of Propositions~\ref{prop:semi_rule} and~\ref{prop:twist_rule}, we build a generating tree for strong-Baxter permutations 
performing local expansions on the right, as illustrated in Figure~\ref{fig:strong}. 
Note that this is possible since removing the last point from any strong-Baxter permutation gives a strong-Baxter permutation. 

Let $\pi$ be a strong-Baxter permutation of size $n$. 
By definition, the active sites of $\pi$ are the $a$'s such that $\pi \cdot a$ is a strong-Baxter permutations. Any non-empty descent (resp. ascent) of $\pi$ is a pair $\pi_i \pi_{i+1}$ such that there exists $\pi_j$ that makes $\pi_j \pi_i \pi_{i+1}$ an occurrence of $2\underbracket[.5pt][1pt]{31}$ (resp. $2\underbracket[.5pt][1pt]{13}$). Then, the non-active sites $a$ of $\pi$ are characterized by the fact that $a \in (\pi_{i+1},\pi_i]$ (resp. $a \in (\pi_i,\pi_j]\,$), for some occurrence $\pi_j \pi_i \pi_{i+1}$ of $2\underbracket[.5pt][1pt]{31}$ (resp. $2\underbracket[.5pt][1pt]{13}$).
Note that in the case where $\pi_{n-1} \pi_n$ is a non-empty descent (resp. ascent), choosing $\pi_j = \pi_{n} +1$ (resp. $\pi_j = \pi_{n}-1$) always gives an occurrence of $2\underbracket[.5pt][1pt]{31}$ (resp. $2\underbracket[.5pt][1pt]{13}$), 
and it is the smallest (resp. largest) possible value of $\pi_j$ for which $\pi_j \pi_{n-1} \pi_n$ is an occurrence of $2\underbracket[.5pt][1pt]{31}$ (resp. $2\underbracket[.5pt][1pt]{13}$).

To the strong-Baxter permutation $\pi$ we assign the label $(h,k)$, where $h$ (resp. $k$) is the number of active sites that are smaller than or equal to (resp. greater than) $\pi_n$. 
As in the proof of Proposition~\ref{prop:semi_rule}, the permutation $1$ has label $(1,1)$, 
and we now need to describe, for $\pi$ of label $(h,k)$, the labels of the permutations $\pi \cdot a$ when $a$ runs over all active sites of $\pi$. 
So, let $a$ be such an active site. 

If $a < \pi_n$, then $\pi \cdot a$ ends with a non-empty descent. 
As in the proof of Proposition~\ref{prop:semi_rule}, all sites of $\pi \cdot a$ in $(a+1,\pi_n+1]$ become non-active 
(due to the avoidance of $\commonpat$). 
Moreover, due to the avoidance of $\twistedpat$, the site immediately above $a$ in $\pi \cdot a$ also becomes non-active. 
All other active sites of $\pi$ remain active in $\pi \cdot a$, hence giving the labels $(i,k)$ for $1 \leq i < h$ in the production of $\Omega_{strong}$
(again, $i$ is such that $a$ is the $i$-th active site from the bottom). 

If $a=\pi_n$, no site of $\pi$ becomes non-active, giving the label $(h,k+1)$ in the production of $\Omega_{strong}$. 

Finally, if $a > \pi_n$, then $\pi \cdot a$ ends with an ascent. 
Because of the avoidance of $\baxpat$, 
we need to consider the occurrences of $2\underbracket[.5pt][1pt]{13}$ in $\pi$ to identify which active sites of $\pi$ become non-active in $\pi \cdot a$. 
It follows from a discussion similar to that in the proof of Proposition~\ref{prop:semi_rule} 
that all sites of $\pi \cdot a$ in $[\pi_n+1,a)$ become non-active. 
Hence, we obtain the missing labels in the production of $\Omega_{strong}$: 
$(h+1,i)$ for $1 \leq i \leq k$ (where $i$ indicates that $a$ is the $i$-th active site from the top). 
\end{proof}

In the same sense that both $\Omega_{Bax}$ and  $\Omega_{TBax}$ specialize $\Omega_{semi}$, 
it is easy to see that  the succession rule $\Omega_{strong}$ is a specialization 
of the rule $\Omega_{Bax}$ (for Baxter permutations) as well as of the rule $\Omega_{TBax}$ (for twisted Baxter permutations). 
In this case, the rule $\Omega_{strong}$ associated with the intersection of these two families is simply obtained by taking, for each object produced, 
the minimum label among the two labels given by $\Omega_{Bax}$ and $\Omega_{TBax}$.
This appears clearly in the following representation:
$$\begin{array}{lllcccccccc}
\Omega_{semi}:&(h,k)&\rightarrow&(1,k+1)&\dots&(h-1,k+1)&(h,k+1)&(h+k,1)&\dots&(h+1,k)\\
\Omega_{Bax}:&(h,k)&\rightarrow&(1,k+1)&\dots&(h-1,k+1)&(h,k+1)&(h+1,1)&\dots&(h+1,k)\\ 
\Omega_{TBax}:&(h,k)&\rightarrow&(1,k)&\dots&(h-1,k)&(h,k+1)&(h+k,1)&\dots&(h+1,k)\\ 
\Omega_{strong}:&(h,k)&\rightarrow&(1,k)&\dots&(h-1,k)&(h,k+1)&(h+1,1)&\dots&(h+1,k).
\end{array}$$ 
This is easily explained. 
Note first that in all four cases $h$ (resp. $k$) records the number of active sites below (resp. above) the rightmost element of a permutation. 
Then, it is enough to remark that among the active sites of a semi-Baxter permutation (avoiding $\commonpat$), 
the avoidance of $\twistedpat$ deactivates only sites above the rightmost element of the permutation, 
while the avoidance of $\baxpat$ deactivates only sites below it. 

\subsection{Generating function of strong-Baxter numbers}
Let $I_{h,k}(x)\equiv I_{h,k}$ denote the generating function for strong-Baxter permutations having label $(h,k)$, with $h,k\geq1$, and let $I(x;y,z)\equiv I(y,z)=\sum_{h,k\geq 1} I_{h,k} y^h z^k$. 
(The notation $I$ stands for Intersection, of the families of Baxter and twisted Baxter permutations.)

\begin{proposition}
The generating function $I(y,z)$ satisfies the following functional equation:
\begin{equation}
\label{inters} 
I(y,z)=xyz+\frac{x}{1-y}(y\,I(1,z)-I(y,z))+xz\,I(y,z)+\frac{xyz}{1-z}(I(y,1)-I(y,z)).
\end{equation}
\end{proposition}

\begin{proof} 
From the growth of strong-Baxter permutations according to $\Omega_{strong}$ we write:
\begin{align*}
I(y,z) &= xyz+x\sum_{h,k\geq 1} I_{h,k} \left( (y+y^2+\dots +y^{h-1})z^{k}+y^h z^{k+1}+y^{h+1}(z+z^2+\dots +z^k) \right) \\
		&= xyz+x\sum_{h,k\geq 1} I_{h,k}\left( \frac{1-y^{h-1}}{1-y} y\, z^{k} +y^h z^{k+1}+  \frac{1-z^{k}}{1-z} y^{h+1}\, z\right) \\
          &= xyz+ \frac{x}{1-y} \left (y\,I(1,z) - I(y,z) \right) + xz\,I(y,z)+\frac{xyz}{1-z} \left ( I(y,1) - I(y,z) \right) \, . \qedhere
\end{align*}
\end{proof}

In order to study the nature of the generating function $I(1,1)$ for strong-Baxter numbers, 
we look at the kernel of eq.~\eqref{inters}, which is 
\begin{equation}
\label{kernel_in}
K(y,z)=1+x\left(\frac{1}{1-y}-z+\frac{yz}{1-z}\right).
\end{equation}
We perform the substitutions $y=1+a$ and $z=1+b$ so that eq.~\eqref{kernel_in} is rewritten as 
\begin{equation}
\label{kerab_in}
K(1+a,1+b)= 1-x Q(a,b) \textrm{ \ where \ } Q(a,b) = \frac{1}{a}+\frac{1}{b}+\frac{a}{b}+a+2+b.
\end{equation}

As in the proof of Theorem~\ref{thm:GFSemiBaxter} (see Subsection~\ref{sec:proofs_GF}), we look for the birational transformations 
$\Phi$ and $\Psi$ in $a$ and $b$ that leave the kernel unchanged, which are: 
$$\Phi:(a,b)\rightarrow\left(a,\frac{1+a}{b}\right),\;\;\;\mbox{ and 
}\;\;\;\Psi:(a,b)\rightarrow\left(-\frac{b}{a(1+b)},b\right).$$
One observes, using Maple for example, that the group generated by these two transformations is not of small order. 
We actually suspect that it is of infinite order, preventing us from using the obstinate kernel method to solve eq.~\eqref{inters}.

Nevertheless, after the substitution $y=1+a$ and $z=1+b$, the kernel we obtain in eq.~\eqref{kerab_in} resembles kernels of functional equations 
associated with the enumeration of families of walks in the (positive) quarter plane \cite{BMM}.

\begin{proposition}\label{prop:walks}
Let $W(t;a,b)$ be the generating function for walks confined in the quarter plane and using 
$\{(-1,0)$, $(0,-1),(1,-1)$, $(1,0)$, $(0,1)\}$ as step set, 
where $t$ counts the number of steps and $a$ (resp. $b$) records the $x$-coordinate (resp. $y$-coordinate) of the ending point. 
The function $W(t;a,b)$ satisfies the following functional equation:
\begin{equation}\label{walks}
W(t;a,b)=1+t\left(\frac{1}{a}+\frac{1}{b}+\frac{a}{b}+a+b\right)W(t;a,b)-\frac{t}{a}W(t;0,b)-t\frac{(1+a)}{b}W(t;a,0).
\end{equation}
\end{proposition}

Not only can we take inspiration from the literature on walks in the quarter plane for our problem of solving eq.~\eqref{inters}, but modifying the step set, we can even arrange that $K(1+a,1+b)$ is exactly the kernel arising in the functional equation for enumerating a family of walks. 

\begin{lemma}\label{lem:strong}
Let $W_2(t;a,b)$ be the generating function for walks confined in the quarter plane and using 
$\{(-1,0),(0,-1),(1,-1),(1,0),(0,1),(0,0), (0,0)\}$ as step (multi-)set, 
where $t$ counts the number of steps and $a$ (resp. $b$) records the $x$-coordinate (resp. $y$-coordinate) of the ending point. 
The difference with the step set of Proposition~\ref{prop:walks} is that 
we have added two copies of the trivial step $(0,0)$, which are distinguished (they can be considered as counterclockwise and clockwise loops for instance). 

The generating functions $W(t;a,b)$ and $W_2(t;a,b)$ are related by 
\begin{equation}
\label{eq:WW2}
W_2(x;a,b) =W\left(\frac{x}{1-2x};a,b\right) \frac{1}{1-2x} 
\end{equation}

Moreover, denoting by $J(x;a,b):=I(x;1+a,1+b)$ the generating function for strong-Baxter permutations, it holds that 
\begin{equation}
\label{eq:JW2}
J(x;a,b) = (1+a)(1+b)\,x\,W_2(x;a,b).
\end{equation}
\end{lemma}

\begin{proof}
First, walks counted by $W_2$ can be described from walks counted by $W$ as follows: 
a $W_2$-walk is a (possibly empty) sequence of trivial steps, followed by a $W$-walk where, after each step, we insert a (possibly empty) sequence of trivial steps. 
This simple combinatorial argument shows that $W_2(x;a,b) =W(\frac{x}{1-2x};a,b) \frac{1}{1-2x}$. 

Next, consider the kernel form of eq.~\eqref{inters} after substituting $y=1+a$ and $z=1+b$, which is 
\begin{equation}\label{comparing1}
(1-xQ(a,b)) J(x;a,b)=x(1+a)(1+b)-x\,\frac{1+a}{a}\,J(x;0,b)-x\,\frac{(1+a)(1+b)}{b}\,J(x;a,0).
\end{equation}
Compare it to the kernel form of eq.~\eqref{walks}:
\begin{equation}\label{comparing2}
(1-t(Q(a,b)-2))W(t;a,b)=1-\frac{t}{a}\,W(t;0,b)-t\,\frac{(1+a)}{b}\,W(t;a,0).
\end{equation}

Substituting $t$ with $\frac{x}{1-2x}$ in eq.~\eqref{comparing2}, and multiplying this equation by $(1+a)(1+b)x$, 
we see that $(1+a)(1+b)\,x\,W_2(x;a,b)$ satisfies eq.~\eqref{comparing1}, proving our claim. 
\end{proof}

With results of~\cite{bostan}, this easily gives the following theorem.

\begin{thm}\label{naturegf}
The generating function $I(1,1)$ of strong-Baxter numbers is not D-finite.
The same holds for the refined generating function $I(a+1,b+1)$. 
\end{thm}

\begin{proof}
Because D-finiteness is preserved by specialization, it is enough to prove that $I(1,1)$ is not D-finite. So, 
with the notation of Lemma~\ref{lem:strong}, our goal is to prove that $J(x;0,0)$ is not D-finite. 
Recall from eq.~\eqref{eq:JW2} that $J(x;a,b) = (1+a)(1+b)\,x\,W_2(x;a,b)$, 
so $J(x;0,0)$ and $W_2(x;0,0)$ coincide up to a factor $x$.
Therefore, proving that $W_2(x;0,0)$ is non D-finite is enough. 

It is proved in~\cite{bostan} that $W(t;0,0)$ is not D-finite. 
Consequently, since $\frac{1}{1-2x}$ and $\frac{x}{1-2x}$ are rational series, 
it follows from eq.~\eqref{eq:WW2} $W_2(x;0,0)$ is not D-finite, as desired.
%
%
\end{proof}

Moreover, some information on the asymptotic behavior of the number of strong-Baxter permutations 
can be derived starting from the connection to walks confined in the quarter plane. In~\cite{bostan} the following proposition is presented.

\begin{proposition}[Denisov and Wachtel, \cite{bostan}(Theorem~4)]\label{DW}
Let $\mathfrak{S}\subseteq\{0,\pm1\}^2$ be a step set which is not confined to a half-plane. 
Let $e_n$ denote the number of $\mathfrak{S}$-excursions of length $n$ confined to the quarter plane $\mathbb{N}^2$ and using only steps in $\mathfrak{S}$.
Then, there exist constants $K$, $\rho$, and $\alpha$ which depend only on $\mathfrak{S}$, such that:
\begin{itemize}
\item if the walk is aperiodic, $e_n\sim K\,\rho^n\,n^\alpha$, 
\item if the walk is periodic (then of period 2), $e_{2n}\sim K\,\rho^{2n}\,(2n)^\alpha,\;\; e_{2n+1}=0$. 
\end{itemize}
\end{proposition}

From~\cite[Section 2.5]{bostan}, the growth constant $\rho_W$ associated with $W(t;0,0)$ is an algebraic number 
whose minimal polynomial is $\mu_{\rho}=t^3+t^2-18t-43$. The approximate value for $\rho_W$ is $4.729031538$. 
We show below that the growth constant of strong-Baxter numbers is closely related to $\rho_W$.

\begin{corollary}\label{cor:strong}
\label{cor:growth_strong}
The growth constant for the strong-Baxter numbers is $\rho_W+2 \approx 6.729031538$.
\end{corollary}

\begin{proof}
From Lemma~\ref{lem:strong}, $I(x;1,1)=x\,W_2(x;0,0)=x\,W(\frac{x}{1-2x};0,0)\,\frac{1}{1-2x}$. 
And from the discussion above, $\frac{1}{\rho_W}$ is the radius of convergence of $W(t;0,0)$. 
The radius of convergence of $g(x)=\frac{x}{1-2x}$ is $\frac{1}{2}$, and $\lim_{{x\to 1/2} \atop {x<1/2}}g(x) = +\infty > \frac{1}{\rho_W}$. 
So, the composition $W(g(x);0,0)$ is supercritical (see \cite[p. 411]{Flaj}), and 
the radius of convergence of $W(\frac{x}{1-2x};0,0)$ is $g^{-1}\left(\frac{1}{\rho_W}\right)=\frac{1}{\rho_W+2}$.
Since $\frac{1}{\rho_W+2}$ is smaller than the radius of convergence $\frac{1}{2}$ of $\frac{1}{1-2x}$, $\frac{1}{\rho_W+2}$ is also the radius of convergence of 
$x\,W(\frac{x}{1-2x};0,0)\,\frac{1}{1-2x} = I(x;1,1)$, proving our claim. 
\end{proof}

\subsection*{Acknowledgements}

The comments of several colleagues on an earlier draft of our paper have helped us improve it significantly. 

First, we would like to thank David Bevan, for sharing his conjectural formulas for $SB_n$ in~\cite{BevanPrivate}, 
for bringing to our attention the conjecture about the enumeration of permutations avoiding $\underbracket[.5pt][1pt]{14}23$, 
and for suggesting the method used in an earlier version of this paper to derive the asymptotic behavior of $SB_n$. 

We also thank Christian Krattenthaler for independently suggesting this method, and pointing to the reference~\cite{MBM_Xin}.

We are very grateful to the referee for their numerous and helpful suggestions. In particular, the current proof of the asymptotic behavior of $SB_n$ 
(together with the reference~\cite{McIntosh}) were suggested to us by the referee.

Finally, we thank Andrew Baxter for clarifying the status of the conjecture about $\underbracket[.5pt][1pt]{14}23$-avoiding permutations, 
which brought reference~\cite{Kasraoui} to our attention.

\end{document}